\documentclass[letterpaper,12pt]{amsart}

\usepackage{amsfonts}
\usepackage{amssymb}
\usepackage{amsthm}
\usepackage{url}
\usepackage{fullpage}
\usepackage{rotating}
\newtheorem{theorem}{Theorem}[section]
\newtheorem{definition}[theorem]{Definition}
\newtheorem{corollary}[theorem]{Corollary}
\newtheorem{lemma}[theorem]{Lemma}
\newtheorem{lemma-definition}[theorem]{Lemma-Definition}
\newtheorem{proposition}[theorem]{Proposition}
\newtheorem{remark}[theorem]{Remark}

\usepackage{amsmath}
\usepackage{latexsym}

\def\Q{\mathbb{Q}}
\def\Qp{Q_p}
\def\Z{\mathbb{Z}}

\def\N{\mathbb{N}}
\def\R{\mathbb{R}}
\def\as#1{\renewcommand\arraystretch{#1}}
\def\comb#1#2{\as{0.8}\left(\!\!\begin{array}{c}#1\\#2
\end{array}\!\!\right)\as{1}}

\def\ff#1{\mathbb{F}_{#1}}

\def\g{{\mathcal G}}
\def\ind{\operatorname{ind}}
\def\ks{K^{\operatorname{sep}}}
\def\lra{\longrightarrow}
\def\m{{\mathfrak m}}
\def\md#1{\ \mbox{\rm(mod }{#1})}
\def\oo{{\mathcal O}}
\def\op{\operatorname}
\def\ord{\operatorname{ord}}
\def\p{\mathfrak{p}}

\def\clK{K^{\operatorname{alg}}}
\def\qpb{\clK}

\def\t{\theta}

\def\ty{\mathbf{t}}
\def\vv{{\mathcal V}}
\def\zpx{\oo[x]}

\newcommand{\disc}{{\operatorname{disc}}}
\newcommand{\quotrem}{{\operatorname{quotrem}}}
\newtheorem{alg}[theorem]{Algorithm}
\newlength{\alginputwidth}
\newlength{\algtmp}
\newcommand{\Algo}[5]
            {
            \begin{alg}[#1] \label{#2}{$\;$}\rm
                \\
\mbox{\enspace}
                \rlap{\rm Input: }\phantom{\rm Output: }
\parbox[t]{\alginputwidth}{#3}
                \\
\mbox{\enspace}
                {\rm Output: }
\parbox[t]{\alginputwidth}{#4}
\parskip0pt
\begin{list}{}{\setlength{\leftmargin}{0pt}}
\item                #5
\end{list}
            \end{alg}
            \goodbreak}

\title{Single-factor lifting and factorization of polynomials over local fields}
\author[Gu\`ardia]{Jordi Gu\`ardia}
\address{Departament de Matem\`atica Aplicada IV, Escola Polit\`ecnica Superior d'Enginyera de Vilanova i la Geltr\'u, Av. V\'\i ctor Balaguer s/n. E-08800 Vilanova i la Geltr\'u, Catalonia}
\email{guardia@ma4.upc.edu}

\author[Nart]{\hbox{Enric Nart}}
\address{Departament de Matem\`{a}tiques,
         Universitat Aut\`{o}noma de Barcelona,
         Edifici C, E-08193 Bellaterra, Barcelona, Catalonia, Spain}
\email{nart@mat.uab.cat}

\author[Pauli]{\hbox{Sebastian Pauli}}
\address{Department of Mathematics and Statistics, University of North Carolina at Greensboro, Greensboro, NC 27402, USA}
\email{s\_pauli@uncg.edu}

\thanks{Partially supported by MTM2009-13060-C02-02 and MTM2009-10359 from the
Spanish MEC}
\date{}

\makeatletter
\@namedef{subjclassname@2010}{%
\textup{2010} Mathematics Subject Classification}


\date{}

\begin{document}
\maketitle

\begin{abstract}
Let $f(x)$ be a separable polynomial over a local field. Montes algorithm computes certain approximations to the different irreducible factors of $f(x)$, with strong arithmetic properties. In this paper we develop an algorithm to improve any one of these approximations,
till a prescribed precision is attained. The most natural application of this ``single-factor lifting" routine is to combine
it with Montes algorithm to provide a fast polynomial factorization algorithm. Moreover, the single-factor lifting algorithm may be applied as well
to accelerate the computational resolution of several global arithmetic problems in which the improvement of an approximation to a single local irreducible factor of a polynomial is required.
\end{abstract}

\section{Introduction}
Polynomial factorization over local fields is an important problem with many applications in
computational number theory and algebraic geometry.
The problem of factoring polynomials over local fields is closely related to several other computational problems, namely the computation of integral bases and the decomposition of
ideals.  Indeed, the factorization algorithms \cite{ford-pauli-roblot,pauli01} implemented in {\tt Pari} \cite{pari} and {\tt Magma} \cite{magma}
are based on the Round Four algorithm \cite{ford} which was originally conceived as an integral bases algorithm.
A similar algorithm was developed by Cantor and Gordon \cite{cantor-gordon}.
All algorithms mentioned above suffer from precision loss in the computation of characteristic polynomials,
which are used in the core part of the algorithm as well as in the lifting of the factorization.


In Montes algorithm \cite{HN,Malgorithm}, originally conceived as an ideal decomposition algorithm \cite{montes},
these precision problems do not exist.
It computes what we call \emph{Montes approximations}  (cf. section \ref{secApprox}) to the irreducible factors of a separable polynomial over a local field,
along with other data needed for the computation of integral bases and ideal factorization,
extremely efficiently.
These approximations can be lifted to an arbitrary precision with
further iterations of Montes algorithm \cite[Sec.4.3]{okutsu}, but the convergence of this method is linear and it is slow in practice.
We present in this paper a \emph{single-factor lifting} algorithm, that
lifts a Montes approximation to an irreducible polynomial to any given precision, with quadratic convergence.


The combination of Montes algorithm and the single-factor lifting algorithm leads to a fast factorization algorithm for polynomials over local fields. For a fixed prime number $p$, this algorithm finds an approximation, with a prescribed precision $\nu\in\N$, to all the irreducible factors of a degree $n$ separable polynomial, $f(x)\in\Z_p[x]$, in $O\left(n^{2+\epsilon}v_p(\disc(f))^{2+\epsilon}+n^2\nu^{1+\epsilon}\right)$ operations with integers less than $p$.

Also, the single-factor lifting algorithm leads to a significant acceleration of the +Ideals package \cite{userguide}. This package contains several routines to deal with fractional ideals in number fields, and it is based on the \emph{Okutsu-Montes representations} of the prime ideals \cite{newapp}. Several of these routines use Montes approximations that need to be improved up to certain precision, and the single-factor lifting brings these routines to an optimal performance.

The outline of the paper is as follows. In section \ref{secTypes} we give an overview of Montes algorithm and the interpretation of its output in terms of Okutsu invariants of the irreducible factors of the input polynomial $f(x)$. Among them, the Okutsu \emph{depth} of each irreducible factor has a strong influence on the computational complexity of $f(x)$. In section \ref{secWidth} we introduce a new Okutsu invariant: the \emph{width} of an irreducible polynomial over a local field. This invariant completes the family of invariants that determine the computational complexity of such an irreducible polynomial: degree, height, index, depth and width. In an Appendix we present families of test polynomials with a controlled variation of all these invariants. We hope that these  polynomials may be useful to test other arithmetic algorithms and detect their strongness and weakness with respect to the variation of each one of these invariants.

In section \ref{secApprox} we discuss how to measure the quality of a Montes approximation, and what arithmetic properties of the irreducible factor we are approximating can be read from a sufficiently good approximation. In section \ref{secImproving} we show that
a Montes approximation can be lifted to an
approximation with arbitrary precision, with quadratic convergence.
In section \ref{secAlgo} we give an algorithm for this lifting procedure and discuss its complexity. Finally, in section \ref{secNumerical}, we present some running times of the factorization algorithm on the families of test polynomials introduced in the Appendix.


\subsection*{Notation}
Throughout the paper we fix a local field $K$, that is, a complete field with respect to a discrete valuation $v$. We let $\oo$ be its ring of integers,
$\m$ the maximal ideal of $\oo$, $\pi\in\m$ a generator of $\m$, $\ff{}=\oo/\m$ the residue class field of $K$, which is suposed to be perfect, and
$^{\raise.8ex\hbox to 8pt{\hrulefill }}\colon \zpx\lra \ff{}[x]$ the natural reduction map.
We write $v\colon \qpb\to \Q\cup\{\infty\}$ for the canonical extension of $v$ to an algebraic closure $\qpb$ of $K$, normalized such that $v(\pi)=1$,
and denote by $\ks\subseteq \qpb$ the separable closure of $K$ in $\qpb$.

Given a field $\mathcal{F}$ and two polynomials $\varphi(y),\,\psi(y)\in \mathcal{F}[y]$,
we denote by $s=\ord_\psi \varphi$ the largest exponent $s$ with $\psi(y)^s\mid \varphi(y)$. Also, we write  $\varphi(y)\sim \psi(y)$ to indicate that there exists a constant
$c\in \mathcal{F}^*$ such that $\varphi(y)=c\psi(y)$.


\section{Complete types and Okutsu invariants}\label{secTypes}

In this section we give an overview of Montes algorithm \cite{HN,Malgorithm} and the interpretation of its output in terms of Okutsu invariants \cite{okutsu}.  Although most of the results about Montes algorithm
are formulated for separable polynomials over the ring of integers of a $p$-adic field, they can be easily generalized
to separable monic polynomials with integral coefficients over local fields with perfect residue field.
In this paper we work in the general setting.
A variant of Montes algorithm formulated for polynomials over locally compact local fields is given in
\cite{pauli10}.

Let $f(x)\in\zpx$ be a monic separable polynomial.
An application of Montes algorithm determines a family of \emph{$f$-complete and optimal types},
that are in one-to-one corres\-pondence to the irreducible factors of
$f(x)$.

Let $\ty$ be the $f$-complete and optimal
type that corresponds to an irreducible factor $F(x)$ of $f(x)\in\oo[x]$. Let $\t\in\ks$ be a root of $F(x)$ and denote $L=K(\t)$. The type $\ty$ has an \emph{order}, which is  a non-negative integer. If $\ty$ has order $0$, then it corresponds to an irreducible factor (say) $\psi(x)$ of $\overline{f}(x)$ over $\ff{}[x]$, that divides $\overline{f}(x)$ with exponent one; in this case $L$ is the unramified extension of $K$ of degree $\deg\psi$. If $\ty$ has order $r\ge 1$, then $\ty$ is structured into $r$ \emph{levels}. At each level $1\le i\le r$, the type stores a monic separable irreducible polynomial
$\phi_i(x)\in\zpx$ and several invariants, that are linked to combinatorial and
arithmetic properties of Newton polygons of higher order of $f(x)$ and capture many properties of the extension $L/K$.  The polynomials $\phi_1,\dots,\phi_r$ are a sequence of approximations to $F(x)$ with
\[
v(\phi_1(\t))<\cdots<v(\phi_r(\t)).
\]
In general we measure the quality of an approximation $\phi(x)$ to $F(x)$
by the valuation $v(\phi(\t))$.

The most important invariants of the type $\ty$ for each level $1\le i\le r$ are:
\medskip
\begin{center}
\begin{tabular}{ll}
$\phi_i(x)\in\zpx$               & a monic irreducible separable polynomial\\
$m_i=\deg \phi_i(x)$             & \\
$\lambda_i=-h_i/e_i$             & where $h_i,e_i$ are positive coprime integers\\
$V_i$               &$=e_{i-1}f_{i-1}(e_{i-1}V_{i-1}+h_{i-1})\in\Z_{\ge0}$\\
$\psi_i(y)\in\ff{i}[y]$          & a monic irreducible polynomial \\
$f_i=\deg \psi_i(y)$             &\\
$\ff{i+1}=\ff{i}[y]/(\psi_i(y))$ &  \\
$z_i$                            & the class of $y$ in $\ff{i+1}$, so that $\psi_i(z_i)=0$
\end{tabular}
\end{center}
\medskip
In the initial step of Montes algorithm the type stores some invariants of level zero, like the monic irreducible factor $\psi_0(y)$ of $\overline{F}(y)$ in $\ff{}[y]$, which is obtained from a factorization of $\overline{f}(y)$.
We set $$e_0=1,\quad h_0=V_0=0,\quad f_0=\deg \psi_0, \quad\ff0=\ff{}, \quad\ff1=\ff0[y]/(\psi_0(y)),$$ and
denote by $z_0\in\ff{1}$ the class of $y$ in $\ff1$. These initial invariants are computed for all types, including those of order $0$.

By construction, the polynomials $\phi_i(x)$ have degree $m_i=(f_0f_1\cdots f_{i-1})(e_1\cdots e_{i-1})$, so that $m_1\mid\cdots\mid m_r$. Note that the fields $\ff{i}$ form a tower of finite extensions of the residue field:
$$
\ff{}=\ff{0}\subseteq \ff{1}\subseteq \cdots \subseteq \ff{r+1},
$$
with $\ff{i+1}=\ff{i}[z_i]=\ff0[z_0,\dots,z_i]$, and $[\ff{i+1}:\ff{0}]=f_0f_1\cdots f_i$.

In each iteration the invariants of a certain level are determined from the data for the previous levels and $f(x)$. Besides the ``physical" invariants, there are other operators determined by the invariants of each level $1\le i\le r$
of the type $\ty$, which are necessary to compute the invariants of the next level:
\medskip
\begin{center}
\begin{tabular}{ll}
$v_i:K(x)\to\Z\cup\{\infty\}$&\ a discrete valuation of the field $K(x)$\\
$N_i\colon K[x]\to 2^{\R^2}$ &\ a Newton polygon operator\\
$R_i\colon \zpx\to \ff{i}[y]$ &\ a residual polynomial operator
\end{tabular}
\end{center}
\medskip
The discrete valuation $v_1$ is the extension of $v$ to $K(x)$ determined by
$$
v_1\colon K[x]\lra\Z\cup\{\infty\},\quad v_1(b_0+\cdots+b_rx^r):=\min\{v(b_j)\mid 0\le j\le r\}.
$$
There is also a $0$-th residual polynomial operator, defined by
$$
R_0\colon \zpx\lra \ff0[y],\quad g(x)\mapsto \overline{g(y)/\pi^{v_1(g)}}.
$$
The Newton polygon operator $N_i$ is determined by the pair $(\phi_i,v_i)$.
For any non-zero polynomial $g(x)\in K[x]$, with $\phi_i$-adic development
\[
g(x)=\sum\nolimits_{s\ge0}a_s(x)\phi_i(x)^s,\quad \deg a_s<m_i,
\]
the polygon $N_i(g)$ is the lower convex hull of the set of points of the plane with coordinates $(s,v_i(a_s(x)\phi_i(x)^s))$.
The negative rational number $\lambda_i$ is the slope of one side of the Newton polygon $N_i(f)$ and
the polynomial $\psi_i(y)$ is a monic irreducible factor of the residual polynomial $R_i(f)(y)$ in $\ff{i}[y]$.

The triple $(\phi_i,v_i,\lambda_i)$ determines the discrete valuation $v_{i+1}$ as follows: for any  non-zero polynomial $g(x)\in K[x]$,
take a line of slope $\lambda_i$ far below $N_i(g)$ and let it shift upwards till it touches the polygon for the first time;
if $H$ is the ordinate of the point of intersection of this line with the vertical axis, then $v_{i+1}(g)=e_i H$.
The invariants $V_i\in\Z_{\ge0}$ are actually: $V_i=v_i(\phi_i)$.

\begin{definition}\label{defs}\rm
Let $\ty$ be a type of order $r\ge 0$ as above, and let $g(x)\in\oo[x]$ be a non-zero polynomial.
\begin{enumerate}
\item
We say that $\ty$ is \emph{optimal} if $m_1<\cdots<m_r$, or equivalently, $e_if_i>1$, for all $1\le i<r$.
\item
We say that $\ty$ is \emph{strongly optimal} if $e_if_i>1$, for all $1\le i\le r$.
\item We define $\ord_\ty(g):=\ord_{\psi_r}(R_r(g))$.
\item
We say that $\ty$ is \emph{$g$-complete} if $\ord_\ty(g)=1$.
\item We say that $g(x)$ is a \emph{representative} of $\ty$ if it is monic of degree $m_{r+1}:=m_re_rf_r$, and $R_r(g)\sim\psi_r$.
In this case, $g(x)$ is irreducible over $\oo$ \cite[Sec.2.3]{HN}\end{enumerate}
\end{definition}

Once an $f$-complete and optimal type $\ty$ is computed,
the main loop of Montes algorithm is applied once more to construct a representative $\phi_{r+1}(x)$ of $\ty$. This polynomial has degree $m_{r+1}=\deg F$ and it is a \emph{Montes approximation} to $F$ (cf. section \ref{secApprox}). Although we keep thinking that $\ty$ has order $r$, actually it supports an $(r+1)$-level with the invariants:
\[
\phi_{r+1}(x),\ m_{r+1}=\deg \phi_{r+1}=\deg F,\ \lambda_{r+1}=-h_{r+1},\ e_{r+1}=1,
\]
$V_{r+1}=e_rf_r(e_rV_r+h_r)=v_{r+1}(\phi_{r+1})$, the discrete valuation $v_{r+1}$ and the field $\ff{r+1}$, which is a computational representation of the residue field of $L$.

The crutial property of $\ty$ is $f$-completeness. By the theorem of the product \cite[Thm.2.26]{HN}, the function $\ord_\ty$ behaves well with respect to multiplication:
$$\ord_\ty(gh)=\ord_\ty(g)+\ord_\ty(h),$$
for any pair of polynomials $g(x),h(x)\in\oo[x]$. Thus, the property $\ord_\ty(f)=1$ singles out an irreducible factor $F(x)$ of $f(x)$ in $\oo[x]$, uniquely determined by $\ord_\ty(F)=1$ and $\ord_\ty(G)=0$, for any other irreducible factor $G(x)$ of $f(x)$. Note that the type $\ty$ is $F$-complete too.

Given a non-zero polynomial $g(x)\in \zpx$, we are usually interested only in the
\emph{principal part} $N_i^-(g)$ of the Newton polygon $N_i(g)$, that is the polygon
$N_i^-(g)$ consisting of the sides of $N_i(g)$ of negative slope.
The \emph{length} of a Newton polygon is by definition the abscissa of the right end point of the polygon.
In the following proposition we recall some more technical facts from \cite{HN} about the invariants introduced above.

\begin{proposition}\label{previous}Let $g(x)\in\oo[x]$ be a non-zero polynomial.

\begin{enumerate}
\item $N_i(F)$ is one-sided of slope $\lambda_i$, for all $1\le i\le r+1$, and $R_i(F)(y)\sim \psi_i(y)^{a_i}$, for some positive exponent $a_{i}$, for all $0\le i\le r$.
\item $N_i(\phi_{i+1})$ is one-sided of slope $\lambda_i$, for all $1\le i\le r$, and
 $R_i(\phi_{i+1})(y)\sim \psi_i(y)$, for all $0\le i\le r$.
\item $\ord_{\psi_i}R_i(g)$ coincides with the length of $N_{i+1}^-(g)$, for all $0\le i\le r$.
\item $v(g(\t))\ge v_{i}(g)/(e_1\cdots e_{i-1})$, for all $1\le i\le r+1$. If $\deg g<m_i$, then equality holds.
\item $v(\phi_i(\t))=(V_i+|\lambda_i|)/(e_1\cdots e_{i-1})$, for all $1\le i\le r+1$.
\end{enumerate}
\end{proposition}

Proposition \ref{previous} (5) is a particular case of the \emph{Theorem of the polygon} \cite[Thm.3.1]{HN}.

There is a natural notion of \emph{truncation} of a type at a certain level. The type $\op{Trunc}_i(\ty)$ is the type of order $i$ obtained by forgetting all levels of order greater than $i$. Note that $\phi_{i+1}(x)$ is a representative of $\op{Trunc}_i(\ty)$, by Proposition \ref{previous} (2).

The \emph{Okutsu depth} of the irreducible polynomial $F(x)$ is the non-negative integer \cite[Thm.4.2]{okutsu}:
$$
\op{depth}(F)=R:=\left\{\begin{array}{ll}
r,  &\mbox{ if }m_r<\deg F,\mbox{ or }r=0,\\
r-1,&\mbox{ if }m_r=\deg F,\mbox{ and }r>0.
\end{array}
\right.
$$
Since, $\deg F/m_r=m_{r+1}/m_r=e_rf_r$, the Okutsu depth of $F$ is equal to $r$ if and only if $e_rf_r>1$; that is, if and only if the type $\ty$ is strongly optimal.
Since $m_1<\cdots<m_{R+1}=\deg F$, we have clearly $R=O(\log(\deg F))$.

The family $[\phi_1,\dots,\phi_R]$ is an \emph{Okutsu frame} of $F(x)$ \cite[Thm.3.9]{okutsu}. This means that for any monic polynomial $g(x)\in\oo[x]$ of degree less than $\deg F$, we have, for all $0\le i\le R$:
\begin{equation}\label{frame}
\dfrac{v(g(\t))}{\deg g}\le\dfrac{v(\phi_i(\t))}{m_i}<\dfrac{v(\phi_{i+1}(\t))}{m_{i+1}} ,\ \mbox{ if }m_i\le \deg g<m_{i+1},
\end{equation}
with the convention that $m_0=1$, $\phi_0(x)=1$.

The numeri\-cal invariants $h_i,\,e_i,\,f_i,\,m_i,\,v(\phi_i(\t))$, for $1\le i\le R$, and the discrete va\-luations
$v_1,\dots,v_{R+1}$ are canonical invariants of $F(x)$ \cite[Cors.3.6+3.7]{okutsu}. They are examples of \emph{Okutsu inva\-riants} of $F(x)$; that is, invariants that can be computed from any Okutsu frame of $F(x)$ \cite[Sec.2]{okutsu}. These invariants
carry on a lot of information about the arithmetic pro\-perties of the extension $L/K$. For instance,
$$
e(L/K)=e_1\cdots e_R=e_1\cdots e_r,\quad f(L/K)=f_0f_1\cdots f_R=f_0f_1\cdots f_r,
$$
and the field $\ff{R+1}=\ff{r+1}$ is a computational representation of the residue field of $L$.

\section{Width of an irreducible polynomial over a local field}\label{secWidth}
Let $F(x)\in\oo[x]$ be a monic irreducible separable polynomial. Let $\t\in \ks$ be a fixed root of $F(x)$, and $L=K(\t)$ the finite separable extension of $K$ determined by $\t$.

In this section we introduce a new Okutsu invariant of an irreducible polynomial over a local field: its \emph{width}. The depth and width of $F(x)$ have a strong influence on the computational complexity of the field $L$, represented as the field extension of $K$ generated by a root of $F(x)$. The relevance of these invariants in a complexity analysis is analogous to that of other parameters more commonly used to measure the complexity of $F$, like the degree, the height (maximal size of the coefficients) and the $v$-value of the discriminant of $F$.

Let $[\phi_1,\dots,\phi_R]$ be an Okutsu frame of $F(x)$. By \cite[Thm.3.5]{okutsu}, there exists an $F$-complete strongly optimal type $\ty_F$ of order $R$, having $\phi_1,\dots,\phi_R$ as its $\phi$-polynomials. Many of the data supported by $\ty_F$ are canonical (Okutsu) invariants of $F$, but the type $\ty_F$ itself is not an intrinsic invariant of $F$.

\begin{lemma}\label{Mrepr}
Let $\ty_F$ be an $F$-complete strongly optimal type of order $R$, and let $\phi_1,\dots,\phi_R$ be its family of $\phi$-polynomials. Let $\phi_{R+1}$ be a representative of $\ty_F$, and take $\phi_0(x):=1$, $m_0:=1$.
For any $1\le i\le R+1$ and any monic polynomial $g(x)\in\oo[x]$ of degree $m_i$, the following conditions are equivalent:
\begin{enumerate}
\item[(a)] $R_{i-1}(g)\sim\psi_{i-1}$.
\item[(b)] $v(g(\t))/m_i>v(\phi_{i-1}(\t))/m_{i-1}$.
\item[(c)] $v(g(\t))>V_i/(e_1\cdots e_{i-1})$.
\end{enumerate}
\end{lemma}

\begin{proof}
Condition (a) says that $g(x)$ is a representative of the truncated type $\op{Trunc}_{i-1}(\ty_F)$. The fact that a representative of a type satisfies (b) was proven in \cite[Lem.3.4]{okutsu}.

Let us write $e:=e_1\cdots e_{i-1}$ for simplicity. Conditions (b) and (c) are equivalent because
$$
\dfrac{V_i}{e}=\dfrac{e_{i-1}f_{i-1}(e_{i-1}V_{i-1}+h_{i-1})}{e}=e_{i-1}f_{i-1}\,\dfrac{V_{i-1}+|\lambda_{i-1}|}{e_1\cdots e_{i-2}}=\dfrac{m_i}{m_{i-1}}v(\phi_{i-1}(\t)),
$$
the last equality by the Theorem of the polygon (Proposition \ref{previous} (5)).

Suppose now that (c) is satisfied.
Since $g$ and $\phi_i$ are both monic of degree $m_i$, the polynomial  $a:=g-\phi_i$ has degree less than $m_i$. By (\ref{frame}), $v(a(\t))\ge \min\{v(g(\t)),v(\phi_i(\t))\}=v(g(\t))>V_i/e$. By Proposition \ref{previous} (4),  $v_i(a)/e=v(a(\t))> V_i/e$; thus, $v_i(a)>V_i$ and $v_i(g)=V_i$. By  \cite[Prop.2.8,(1)]{HN}, $R_{i-1}(g)=R_{i-1}(\phi_i)$, and this implies  $R_{i-1}(g)\sim \psi_{i-1}$, by Proposition \ref{previous} (2).
\end{proof}

\begin{definition}\label{setsG}\rm
For each $1\le i\le R+1$, let $\g_i\subseteq \oo[x]$ be the set of all monic polynomials of degree $m_i$ satisfying any of the conditions of Lemma \ref{Mrepr}. As mentioned along the proof of the lemma, the polynomials in $\g_i$ are the representatives of the truncated type $\op{Trunc}_{i-1}(\ty_F)$; thus, they are all irreducible over $\oo$. In particular, $\g_{R+1}$ is the set of representatives of $\ty_F$.
\end{definition}

Actually, $m_i$ is the minimal degree of a polynomial satisfying condition (a) \cite[Sec.2.3]{HN}.  For $i\le R$, (\ref{frame}) shows that the value $v(\phi_i(\t))$ is maximal among all polynomials in $\g_i$:
$$
v(\phi_i(\t))\ge v(g(\t)),\ \forall g\in\g_i, \  1\le i\le R.
$$

Since the rational numbers $v(\phi_1(\t)),\dots v(\phi_R(\t))$ are Okutsu invariants of $F(x)$, the sets of polynomials $\g_1,\dots,\g_{R+1}$, and their sets of values
$$\vv_i:=\left\{v(g(\t))\mid g\in \g_i\right\}\subseteq \Q\cup \infty,\quad 1\le i\le R+1,
$$ are intrinsic invariants of $F(x)$ too. The sets $\vv_1,\dots,\vv_R$ are finite, because they are bounded (by (\ref{frame})) and $v$ is a discrete valuation. However, $\vv_{R+1}$ is an infinite set that contains $\infty$, because $F$ clearly belongs to $\g_{R+1}$.

\begin{definition}\label{width}The \emph{width} of $F(x)$ is the vector of non-negative integers:
$$
\op{width}(F):=(\#\vv_1,\dots,\#\vv_R).
$$
\end{definition}

Our next aim is to show that $\op{width}(F)$ is an Okutsu invariant of $F$ and to compute it in terms of the Okutsu frame $[\phi_1,\dots,\phi_R]$.

\begin{proposition}\label{compwidth}
With the above notation, for each $1\le i\le R$ we have:
$$
\#\vv_i=\lceil|\lambda_i|\rceil=\lceil h_i/e_i\rceil.
$$
\end{proposition}

\begin{proof} Let us denote $e:=e_1\cdots e_{i-1}$ for simplicity.

Any $g(x)\in\g_i$ is a representative of the type $\op{Trunc}_{i-1}(\ty)$, and we saw along the proof of Lemma \ref{Mrepr} that
$v_i(g)=V_i$ is constant.
The Theorem of the polygon \cite[Thm.3.1]{HN}, applied to both polynomials, shows that
\begin{equation}\label{compare}
\begin{array}{l}
v(g(\t))=(v_i(g)+|\lambda|)/e=(V_i+|\lambda|)/e,\\
v(\phi_i(\t))=(V_i+|\lambda_i|)/e.
\end{array}
\end{equation}
where $\lambda$ is the slope of the one-sided Newton polygon of $i$-th order $N_{g,v_i}(F)$, computed with respect to $g(x)$ and $v_i$.

By \cite[Thm.3.1]{Malgorithm}, the property $e_if_i>1$ implies that
$\phi_i\in\g_i$ is an optimal representative of  $\op{Trunc}_{i-1}(\ty_F)$; more precisely, this theorem shows that
\begin{equation}\label{bordeaux}
 |\lambda|\le |\lambda_i|, \ \mbox{ and }\
 |\lambda|<|\lambda_i|\,\Longrightarrow\,\lambda\in\Z.
\end{equation}
Hence, (\ref{compare}) and (\ref{bordeaux}) prove that $\#\vv_i\le\lceil|\lambda_i|\rceil$.

In order to prove the opposite inequality, let us show that for any given integer $0< d<|\lambda_i|$, there is a monic polynomial $g\in\oo[x]$ of degree $m_i$ such that  $v(g(\t))=(V_i+d)/e$. Note that such a polynomial belongs to $\g_i$ because it satisfies (c) of Lemma \ref{Mrepr}.
The idea is to spoil the optimal polynomial $\phi_i\in\g_i$, by adding an adequate term: $g(x)=\phi_i(x)+a(x)$, leading to the desired value of $v(g(\t))$. It is sufficient to take $a(x)\in\oo[x]$ satisfying
\begin{equation}\label{via}
\deg a<m_i,\quad v_i(a)=V_i+d.
\end{equation}
In fact, by Proposition \ref{previous} (4), $v(a(\t))=v_i(a)/e=(V_i+d)/e<v(\phi_i(\t))$, so that $g(x)=\phi_i(x)+a(x)$ is monic of  degree $m_i$ and has value: $v(g(\t))=v(a(\t))$.

The existence of $a(x)\in\oo[x]$ satisfying (\ref{via}) is guaranteed by
\cite[Prop.2.10]{HN}, as long as $V_i+d\ge e_{i-1}f_{i-1}v_i(\phi_{i-1})$. By \cite[Thm.2.11]{HN}, we have $e_{i-1}f_{i-1}v_i(\phi_{i-1})=V_i$, so that the desired inequality is obvious.
\end{proof}

The depth of $F$ is linked to the degree: $R=O(\log(\deg F))$, but is is a finer invariant. It is easy to construct irreducible polynomials having the same (large) degree, analogous height and the same $v$-value of the discriminant, but prescribed different depths, from $R=1$ to $R=\lfloor \log_2(\deg F)\rfloor$. A sensible-to-depth algorithm solving some arithmetic task concerning these polynomials will be much faster  for the polynomials with small depth.

In the same vein, the width of $F$ is linked to $v(\disc(F))$, but it is a finer invariant. More precisely, the width is directly linked to the index $\ind(F)$, which is defined as the length of $\oo_L/\oo[\t]$ as an $\oo$-module, and it satisfies: $v(\disc(F))=v(\disc(L))+2\ind(F)$ . The following formula for the index shows the connection between index and width.

\begin{proposition}\label{indextheorem}
$$
\ind(F)=\dfrac{\deg F}2\sum_{1\le i\le R}\dfrac1{e_1\cdots e_{i-1}}\left(|\lambda_i|\left(\frac{\deg F}{m_i}-1\right)-\frac{e_i-1}{e_i}\right).
$$
\end{proposition}

\begin{proof}We keep the above notation for $\ty_F$ and $[\phi_1,\dots,\phi_R]$. The Newton polygons $N_i(F)$, for $1\le i\le R$, are all one-sided of slope $\lambda_i$. The length of the projections of $N_i(F)$ to the horizontal and vertical axis are $E:=\deg F/m_i$ and $H:=|\lambda_i|E$, respectively. By the Theorem of the index \cite[Thm.4.18]{HN}, $\ind(F)=\ind_1(F)+\cdots+\ind_R(F)$, where $\ind_i(F)$ is $f_0\cdots f_{i-1}$ times the index of the side $N_i(F)$; that is \cite[Def.4.12]{HN}:
$$
\ind_i(F)=\dfrac{f_0\cdots f_{i-1}}2\left(|\lambda_i|E^2-|\lambda_i|E-E+\dfrac{E}{e_i}\right).
$$
Since $E=(e_if_i)\cdots (e_Rf_R)$, clearly $f_0\cdots f_{i-1}E=\deg F/(e_1\cdots e_{i-1})$, and $\ind_i(F)$ coincides with the $i$-th term of the sum in the statement of the proposition.
\end{proof}

By using the techniques of \cite[Sec.2.3]{HN}, it is easy to construct irreducible polynomials of fixed depth $R$, and
prescribed values of all invariants $e_1,\dots,e_R$, $f_0,\dots,f_R$, $h_1,\dots,h_R$. Since the degree depends only on the $e_i$ and $f_i$ invariants, whereas the slopes $\lambda_i$ depend on $e_i$ and $h_i$, we may construct polynomials with the same degree, depth and index, but different width. Again, sensible-to-width algorithms solving arithmetic tasks concerning these polynomials will be much faster for the polynomials with small width.

Unfortunately, it is difficult to take into account these invariants in theoretical analysis of complexity. For instance, we have not been able to do this in the analysis of the single-factor lifting algorithm in section \ref{secAlgo}. Thus, we thought it might be interesting to test numerically the sensibility of the algorithm to as many complexity parameters as possible, including the depth and width of the irreducible factors of the input polynomial. To this end, in an appendix we present families of test polynomials that, besides the classical parameters, present a controlled variation of the number of irreducible factors and the depth and width of each factor. In section \ref{secNumerical}, we present running times of the factorization of some of these test polynomials, obtained by applying Montes algorithm followed by the single-factor lifting algorithm for each of the irreducible factors. The numerical data suggest that this factorization algorithm is sensible to both invariants, depth and width.


\section{Montes approximations}\label{secApprox}
We go back to the situation of section \ref{secTypes}. We take an $f$-complete optimal type $\ty$ of order $r$, that singles out a (never computed) monic irreducible factor $F(x)\in\oo[x]$ of the monic separable polynomial $f(x)\in\oo[x]$. Let $\t\in \ks$ be a fixed root of $F(x)$, $L=K(\t)$ the finite separable extension of $K$ determined by $\t$, and $\oo_L$ the ring of integers of $L$. Let $R$ be the Okutsu depth of $F$, and consider the family of canonical sets, $\g_1,\dots,\g_{R+1}$, introduced in Definition \ref{setsG}.

In this section we deal with approximations to $F$. We discuss how to measure the quality of the approximations and the arithmetic properties of $L/K$ that can be derived from any sufficiently good approximation.

\begin{definition}\label{Mapprox}\rm
The polynomials in the set $\g_{R+1}$ are called \emph{Okutsu approximations} to $F(x)$.
\cite[Sec.4]{okutsu}.

The representatives of the type $\ty$ are called \emph{Montes approximations} to $F(x)$.
\end{definition}

The concept of Okutsu approximation to $F(x)$ is intrinsic
(depends only on $F(x)$), and ``being an Okutsu approximation to'' is an equivalence relation on the set of irreducible polynomials in $\oo[x]$ \cite[Lem.4.3]{okutsu}.

However, a Montes approximation is an object attached to $F(x)$ \emph{as a factor of $f(x)$}. Hence, it depends on $f(x)$ and it has no sense to interpret it as a binary relation between irreducible polynomials.

\begin {remark}\rm
Suppose a factorization algorithm is designed in such a way that approximations $\phi$ to a certain irreducible factor $F$ of $f(x)$ are constructed, and the iteration steps consist of finding, for a given $\phi$, a better approximation $\Phi$ satisfying $v(\phi(\t))<v(\Phi(\t))$. Then, by their very definition, the depth and width of $F$ measure the obstruction that the algorithm encounters to reach an Okutsu approximation (for the first time). More precisely, the sum of the components of the width are an upper bound for the number of iterations. Also, the fact that the width is graduated by the depth makes sense because it is highly probable that the iterations at a higher depth will have a higher cost.
\end {remark}

\begin{lemma}
A Montes approximation is always an Okutsu approximation. The converse holds if and only if $R=r$.
\end{lemma}

\begin{proof}
If $R=r$, then the type $\ty$ is strongly optimal and the two concepts coincide. In fact, $\ty$ is always $F$-complete ($\ord_\ty(F)=1$), and Lemma \ref{Mrepr} shows that $\g_{R+1}$ is the set of representatives of $\ty$.

Suppose $R=r-1$, and let $\phi_{r+1}$ be a Montes approximation to $F$. The degree of $\phi_{r+1}$ is $m_{r+1}=m_r=m_{R+1}=\deg F$. By the Theorem of the polygon,
$$
v(\phi_{r+1}(\t))>\dfrac{V_{r+1}}{e_1\cdots e_r}=
\dfrac{e_rf_r(e_rV_r+h_r)}{e_1\cdots e_r}>\dfrac{V_r}{e_1\cdots e_{r-1}},
$$
because $h_r>0$. Therefore, $\phi_{r+1}$ satisfies condition (c) of Lemma \ref{Mrepr} for $i=R+1=r$, and it belongs to $\g_{R+1}$.
On the other hand, the polynomial $\phi_r=\phi_{R+1}$ is an Okutsu approximation to $F(x)$, but it is not a representative of $\ty$. In fact, the Newton polygon $N_r(\phi_r)$ is the single point $(1,V_r)$; thus, the residual polynomial $R_r(\phi_r)$ is a constant, and $\psi_r\nmid R_r(\phi_r)$.
\end{proof}

One cannot expect to deal only with strongly optimal types. For instance, if the polynomial $f(x)$ has
different irreducible factors that are Okutsu approximations to each other; these irreducible factors have the same Okutsu frames \cite[Lem.4.3]{okutsu} and hence the same strongly optimal types attached to them \cite[Thms.3.5+3.9]{okutsu}. Therefore, in order to distinguish them it is necessary to consider non-strongly optimal types. In other words, once we reach an Okutsu approximation $\phi_{R+1}$ to $F$, it may happen that $\phi_{R+1}$ is also an Okutsu approximation to other irreducible factors of $f(x)$; thus, it is necessary to go one step further and compute a Montes approximation $\phi_{R+2}=\phi_{r+1}$ to $F$, that singles out this irreducible factor.
This property suggests that a Montes approximation is the right object to start with for a single-factor lifting algorithm, aiming to improve a given approximation to $F$ till a prescribed precision is attained.

%

\subsection*{Measuring the quality of approximations}
For simplicity we set from now on:
$$
e:=e(L/K)=e_1\cdots e_r,\quad w:=v_{r+1}, \quad V:=w(\phi_{r+1})=V_{r+1}.
$$

The following result is an immediate consequence of \cite[Thm.2.11+Thm.3.1]{HN}.

\begin{lemma}\label{eqwphi}
Let $\Phi$ be a Montes approximation to $F$. By Proposition \ref{previous} (3),
 the principal polygon $N_{\Phi,w}^-(f)$ has length one, so that the slope
$-h_\Phi$ of its unique side is a negative integer (see Figure \ref{fignewton}).
We have $w(\Phi)=V$ and $v(\Phi(\t))=\left(V+h_\Phi\right)/e$.
\end{lemma}

As mentioned above, $v(\Phi(\t))$ is a measure of the
quality of the approximation; hence, the integer $h_\Phi$ is the relevant invariant to measure the precision of $\Phi(x)$ as an approximation to $F(x)$. Actually, $h_\Phi$ is the ideal invariant to look at, because it is also explicitly linked to an estimation of $v_1(F(x)-\Phi(x))$, which is the traditional value to measure the precision of an approximation.

\begin{lemma}[{\cite[Lem.4.5]{okutsu}}]\label{lem45okutsu}
Let $\Phi(x)\in\zpx$ be a Montes approximation to $F(x)$
and let $-h_\Phi$ be the slope of the principal polygon $N_{\Phi,w}^-(f)$.  Then
\[
F(x)\equiv \Phi(x)\md{\m^{\lceil\nu\rceil}},
\]
where $\nu=\nu_0+(h_\Phi/e)$
and $\nu_0$ is the (constant) rational number
\begin{equation}\label{nu0}
\nu_0:=\dfrac{h_1}{e_1}+\dfrac{h_2}{e_1e_2}+\cdots+\dfrac{h_r}{e_1\cdots e_r}.
\end{equation}
\end{lemma}

Thus, when we replace $\Phi$ by successive (better) approximations to $F(x)$,
the improvement of the precision is determined by the growth of the parameter $h_\Phi$.


\subsection*{Common arithmetic properties of Montes approximations}
Let $\Phi(x)$ be a Montes approximation to $F$.
Fix $\beta\in\ks$, a root of $\Phi$, and consider $N=K(\beta)$, $\oo_N$ the ring of integers of $N$ and $\m_N$ its maximal ideal.

Since $F$ and $\Phi$ are representatives of $\ty$, we have:
$$
\deg F=\deg\Phi,\quad w(F)=w(\Phi)=V,\quad \ord_\ty(F)=\ord_\ty(\Phi)=1.
$$
By  \cite[Lem.4.3]{okutsu} the Okutsu frame $[\phi_1,\dots,\phi_R]$ of $F(x)$ is also an Okutsu frame of $\Phi(x)$.
Therefore the two polynomials $F(x)$ and $\Phi(x)$ have the same Okutsu invariants.
In particular, the extensions $N/K$ and $L/K$ have the same ramification index and residual degree:
$$e(L/K)=e(N/K),\quad f(L/K)=f(N/K).
$$
Actually, as shown in \cite{Ok}, $L/K$ and $N/K$ have isomorphic maximal tamely ramified subextensions \cite[Cor.2.9]{okutsu}. Also, Proposition \ref{indextheorem}
 shows that $\ind(F)=\ind(\Phi)$.

The field $\ff{r+1}$ is a common computational representation of the residue fields of $L/K$ and $N/K$. More precisely, certain rational functions $\gamma_i(x)\in K(x)$, that depend only on the type $\ty$ \cite[Sec.2.4]{HN}, determine an explicit isomorphism,
\begin{equation}\label{gamma}
\gamma\colon\ff{r+1}\longrightarrow \oo_N/\m_N,\quad z_0\mapsto \overline{\beta},\,z_1\mapsto \overline{\gamma_1(\beta)},\,\dots,\,z_r\mapsto \overline{\gamma_r(\beta)}.
\end{equation}And we get a completely analogous isomorphism $\gamma\colon\ff{r+1}\longrightarrow \oo_M/\m_M$, just by replacing $\beta$ by $\alpha$.

The exponent of $F(x)$ is by definition the least non-negative integer $\exp(F)$ such that $$\pi^{\exp(F)}\oo_L\subseteq \oo[\t].$$
An explicit formula for $\exp(F)$ can be given in terms of the Okutsu invariants:
\begin{theorem}[{\cite[Thm.5.2]{newapp}}]\label{theoexp}
The exponent of $F(x)$ is
 $\exp(F)=\lfloor\mu_F\rfloor$, where
$$
\mu_F:=\dfrac{V}{e}-\nu_0=\sum_{i=1}^R(e_if_i\cdots e_Rf_R-1)\dfrac{h_i}{e_1\cdots e_i},
$$
and $\nu_0$ is the constant from equation (\ref{nu0}).
\end{theorem}

Therefore, the polynomials $F$ and $\Phi$ have the same exponent too.
Moreover, all results of \cite{HN,Malgorithm,okutsu} that relate arithmetic properties of the extension $L/K$ with the invariants stored by the type $\ty$, can be equally applied to link $\ty$ with arithmetic properties of the extension $N/K$.
For instance, we shall frequently use the following remarks, that follow from Proposition \ref{previous} (4) and \cite[Lem.2.17 (1)]{HN}.

\begin{lemma}\label{wv}
Let $\Phi(x)\in\zpx$ be a Montes approximation to $F$, and take $\beta\in\ks$ a root of $\Phi$. Let $P(x)\in K[x]$ be an arbitrary polynomial.
\begin{enumerate}
\item If $\deg P< \deg F$, then $v(P(\beta))=w(P)/e=v(P(\t))$.
\item If $P(x)=\sum_{0\le s}a_s(x)\Phi(x)^s$ is the canonical $\Phi$-adic development of $P$, then $w(P)=\min_{0\le s}\{w(a_s\Phi^s)\}$.
\end{enumerate}
\end{lemma}

In the lifting algorithm we will need to construct a polynomial $\Psi(x)\in K[x]$ such that $\deg\Psi<\deg F$ and $w(\Psi)$ has a given value. To this end we can use \cite[Algorithm 14]{pauli10}.

\begin{lemma}\label{lemPsi}
Let $m=\deg F$, $u\in\Z$, and $R$ the Okutsu depth of $F$.
There is an algorithm that finds exponents $j_\pi\in\Z$ and $j_1,\dots,j_R\in\N$ such that
\[
\Psi(x)=\pi^{j_\pi}\phi_1(x)^{j_1}\cdot\dots\cdot\phi_R(x)^{j_R}
\]
has degree less than $m$ and $w(\Psi)=u$, in $O((\log m)^3)$ operations of integers less than $m$.
\end{lemma}

For the commodity of the reader we reproduce the algorithm. First, we express $u=Ne+t$, $0\le t <e$. Then, the routine shown below computes $j_1,\dots,j_R$ and an integer $M$. Finally one takes $j_\pi=N+M$.\medskip

\noindent{\bf Universal polynomial routine}\medskip

 $j_R\leftarrow h_R^{-1}t \bmod e_R$

 $M\leftarrow (t-j_Rh_R)/e_R$

 For $i=R$ to $2$ by $-1$ do

\qquad\qquad$j_{i-1}\leftarrow h_{i-1}^{-1}(M-j_{i}V_{i}) \bmod e_{i-1}$

\qquad\qquad$M\leftarrow(M-j_iV_i-j_{i-1}h_{i-1})/e_{i-1}$\medskip

Along the process of improving the Montes approximations to $F$, the required value of $w(\Psi)$ remains constant. By Lemma \ref{wv}, the value $v(\Psi(\beta))=w(\Psi)$ remains constant too: it does not depend on the pair $(\Phi,\beta)$. Hence, $\Psi$ is a kind of universal polynomial that is computed only once as an initial datum, and used in all iterations.





\section{Improving a Montes approximation}\label{secImproving}
We keep all notation of section \ref{secApprox}, and we denote from now on $m:=\deg F=[L\colon K]$.

The aim of this section is to find a quadratic convergence iteration method to improve the Montes approximations to
$F(x)$. More precisely, given a Montes approximation $\phi(x)$, we shall construct another Montes
approximation $\Phi(x)$ such that $h_\Phi\ge 2h_\phi$, where $h_\Phi$ and $h_\phi$ are the slopes of the Newton polygons
$N^-_{\Phi,w}(f)$ and $N^-_{\phi,w}(f)$, respectively.

The general idea of the lifting method is inspired in the classical Newton iteration method.
Instead of Taylor development of $f(x)$, we consider its $\phi$-adic development:
$$
f(x)=\sum_{0\le s}a_s(x)\phi(x)^s,\quad \deg a_i<m.
$$
The principal Newton polygon $N_{\phi,w}^-(f)$ has length one,
as illustrated in Figure \ref{fignewton}. Lemma \ref{wv},(2) shows that
$w(f)=\min_{0\le s\le m}\{w(a_s\phi^s)\}=w(a_1\phi)$.
Therefore, for all $s\ge 2$, Lemma \ref{wv},(1) shows that:
\begin{equation}\label{eqvalnewt}
v(a_1(\t)\phi(\t))=\dfrac{w(a_1)}e+v(\phi(\t))\le\dfrac{w(a_s\phi^{s-1})}e+v(\phi(\t))< v(a_s(\t)\phi(\t)^s),
\end{equation}
the last inequality because $w(\phi)/e=V/e<v(\phi(\t))$ by the Theorem of the polygon.
If we evaluate the $\phi$-adic development at $\t$ we obtain
\[
\phi(\t)+\frac{a_0(\t)}{a_1(\t)}=-\frac{\sum_{2\le s}a_s(\t)\phi(\t)^s}{a_1(\t)}.
\]
With (\ref{eqvalnewt}) we get
\[
v\left(\phi(\t)+\frac{a_0(\t)}{a_1(\t)}\right)=v\left(\frac{\sum_{2\le s}a_s(\t)\phi(\t)^s}{a_1(\t)}\right)>v\left(\phi(\t)\right).
\]
As $\phi(x)$ is irreducible we can use the extended Euclidean algorithm to obtain
$a_1^{-1}(x)\in K[x]$ with $a_1(x)a_1^{-1}(x)\equiv 1\mod \phi(x)$.
For $\Phi(x):=\phi(x)+A(x)$ where $A(x)\equiv a_0(x)a_1^{-1}(x)\mod\phi(x)$, with $\deg A<\deg\phi$, we get
\[
\frac{V+h_\Phi}{e}=v(\Phi(\t))=v(\phi(\t)+A(\t))>v(\phi(\t))=\frac{V+h_\phi}{e}.
\]
Thus $h_\Phi>h_\phi$ and $\Phi(x)$ is a better approximation to the irreducible factor $F(x)$ of $f(x)$.

\begin{figure}\caption{Newton polygon $N_{\phi,w}(f)$ where $f(x)=a_0(x)+a_1(x)\phi+\dots$ is
the $\phi$-adic expansion of $f(x)$ and $\phi(x)$ is a Montes approximation to $F(x)$.}\label{fignewton}
\setlength{\unitlength}{6mm}
\begin{picture}(13,10)
\put(4.85,7.8){$\bullet$}\put(6.85,3.8){$\bullet$}\put(8.85,3.8){$\bullet$}
\put(4,1){\line(1,0){7}}\put(5,0){\line(0,1){9}}
\put(5,8){\line(1,-2){2}}\put(5.02,8){\line(1,-2){2}}
\put(7,4){\line(1,0){2}}\put(7,4.01){\line(1,0){2}}
\put(9,4){\line(2,1){1.2}}\put(9.02,4){\line(2,1){1.4}}
\multiput(7,0.9)(0,.25){13}{\vrule height2pt}
\multiput(4.9,4)(.25,0){9}{\hbox to 2pt{\hrulefill }}
\put(6.1,6){$-h_\phi$}
\put(3.0,7.8){$w(a_0)$}
\put(0.3,3.8){$w(f)=w(a_1\phi)$}
\put(6.85,0.3){$1$}
\put(5.1,0.3){$0$}
\put(8.6,6){$N_{\phi,w}(f)$}
\end{picture}
\end{figure}
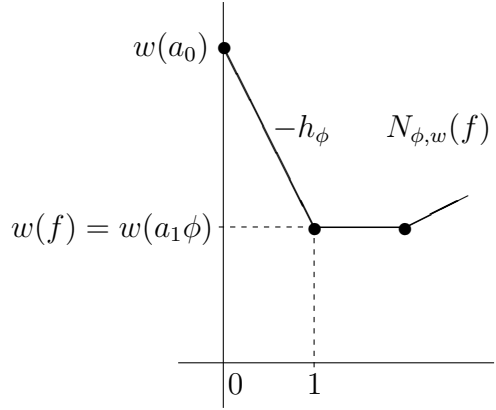

In the following we show that, as in the classical method, the measure of the approximation is doubled in each iteration: $h_\Phi\ge 2h_\phi$; thus, we are led to a quadratic convergence algorithm.
A crucial point for efficiency is to avoid the inversion of $a_1(\t)$ in $L$. To this end, we demonstrate that classical Newton lifting yields a more efficient way for finding an approximation to the polynomial $a_1^{-1}(x)$
and that in each iteration only one Newton lifting step is needed.

\subsection{The main theorem: doubling the slope}\label{subsecmain}
Let $\phi(x)\in\zpx$ be a given Montes approximation to the irreducible factor $F(x)$ of $f(x)$.
We choose a root $\alpha\in\ks$ of $\phi(x)$ and consider the field $M=K(\alpha)$
with ring of integers $\oo_M$ and maximal ideal $\m_M$.

The next theorem gives a criterion to ensure that the slope $h_\phi$ is (at least) doubled if we take a Montes approximation of the form $\Phi(x)=\phi(x)-A(x)$, for an adequate polynomial $A(x)$ of degree less than $m$.

\begin{theorem}\label{main}
Let $\phi$ be a Montes approximation to $F$, and let $h$ be a positive integer, $0<h\le h_\phi$.
For any polynomial $A(x)\in\zpx$ of degree less than $m=\deg F$, the following conditions are equivalent:
\begin{enumerate}
\item $v\left(a_0(\alpha)+a_1(\alpha)A(\alpha)\right)\ge (w(f)+2h)/e$,
\item The polynomial $\Phi(x):=\phi(x)-A(x)$, is a Montes approximation to $F(x)$, and $N_{\Phi,w}^-(f)$ is one-sided of slope $-h_\Phi$, with  $h_\Phi\ge 2h$,
\item $v\left(a_0(\t)+a_1(\t)A(\t)\right)\ge (w(f)+2h)/e$.
\end{enumerate}
\end{theorem}

\begin{proof}
By the shape of $N_{\phi,w}^-(f)$, we know that $w(a_0)\ge w(f)+h$, $w(a_1)=w(f)-w(\phi)=w(f)-V$. Since $\deg a_0,\,\deg a_1<m$, Lemma \ref{wv} shows that:
\begin{equation}\label{vai}\as{1.2}
\begin{array}{l}
 v(a_0(\t))=v(a_0(\alpha))=w(a_0)/e\ge(w(f)+h)/e,\\
 v(a_1(\t))=v(a_1(\alpha))=w(a_1)/e=(w(f)-V)/e.
\end{array}
\end{equation}
Let $\Phi(x):=\phi(x)-A(x)$. From the $\phi$-adic development of $f(x)$ we get the $\Phi$-expansion
\[
f(x)=\sum\nolimits_{0\le s}a_s\phi^s=\sum\nolimits_{0\le s}a_s(\Phi+A)^s=\sum\nolimits_{0\le s}b_s\Phi^s,
\]
where
\begin{align*}
 b_0&=a_0+a_1A+\cdots+a_sA^s+\cdots+a_nA^n,\\
 b_1&=a_1+2a_2A+\cdots+sa_sA^{s-1}+\cdots,\\
    \vdots\\
 b_k&=a_k+(k+1)a_{k+1}A+\cdots+\comb{s}{k}a_sA^{s-k}+\cdots.
\end{align*}
We shall see along the proof of the theorem that each of the conditions (1), (2), and (3) implies that
$$w(A)\ge V+h,  $$
which in turn implies $w(\Phi)=V$.
For all $s\ge k\ge0$ we obtain the lower bound
$$ w\left(\comb{s}{k}a_sA^{s-k}\Phi^k\right)\ge w(a_s\phi^s)+(s-k)h\ge w(f)+(s-k)h,$$
from which we deduce:
\begin{equation}\label{bounds}
\begin{array}{ll}
w(a_2A^2+\cdots+ a_nA^n)\ge w(f)+2h,&(s\ge 2, k=0)\\
w\left((2a_2A+\cdots+n a_nA^{n-1})\Phi\right)\ge w(f)+h,&(s>k=1)\\
w(b_1\Phi)=w(a_1\Phi)=w(a_1\phi)=w(f),&\\
w(b_k\Phi^k)\ge w(f),\ \forall\,k>1,&(s\ge k>1)
\end{array}
\end{equation}
We consider the canonical $\Phi$-adic developments:
$$
\begin{array}{l}
a_0+a_1A=d_0+d_1\Phi,\\
a_2A^2+\cdots +a_nA^n=u_0+u_1\Phi+\cdots +u_s\Phi^s+\cdots,\\
(2a_2A+\cdots+n a_nA^{n-1})\Phi=g_1\Phi+g_2\Phi^2+\cdots+g_s\Phi^s+\cdots.
\end{array}
$$
The bounds (\ref{bounds}) and Lemma \ref{wv} (2) show that:
$$
\begin{array}{l}
w(f)+h\le w\left(a_0+a_1A\right)=\min\{w(d_0),\,w\left(d_1\Phi\right)\},\\
w(f)+2h\le w\left(a_2A^2+\cdots a_nA^n\right)=\min_{0\le s}\{w\left(u_s\Phi^s\right)\},\\
w(f)+h\le w\left((2a_2A+\cdots+n a_nA^{n-1})\Phi\right)=\min_{1\le s}\{w\left(g_s\Phi^s\right)\}.
\end{array}
$$
Hence,
\begin{equation}\label{bounds2}
\begin{array}{l}
w(d_1)\ge w(f)+h-V,\\
w(u_0)\ge w(f)+2h,\quad w(u_1)\ge w(f)+2h-V,\\
w(g_1)\ge w(f)+h-V.
\end{array}
\end{equation}
We now prove that condition (1) implies condition (2). From
$$
v(a_0(\alpha))\ge (w(f)+h)/e, \quad v\left(a_0(\alpha)+a_1(\alpha)A(\alpha)\right)\ge (w(f)+2h)/e,
$$
we deduce $v\left(a_1(\alpha)A(\alpha)\right)\ge (w(f)+h)/e$. By (\ref{vai}), we get
$$
v\left(A(\alpha)\right)\ge(w(f)+h)/e-(w(f)-V)/e=(V+h)/e.
$$
Since $\deg A<m$, Lemma \ref{wv} shows that $w(A)\ge V+h$, so that all bounds (\ref{bounds}), (\ref{bounds2}) hold. Also, $v\left(A(\t)\right)\ge(V+h)/e$. By  the theorem of the polygon, $v(\phi(\t))>V/e$. Hence, $v\left(\Phi(\t)\right)> V/e$, and $\Phi$ is a Montes approximation to $F$, by Lemma \ref{Mrepr}.

In particular, the shape of the Newton polygon $N_{\Phi,w}^-(f)$ is analogous to the shape of
$N_{\phi,w}^-(f)$ (see Figure \ref{fignewton}).
Thus, condition (2) holds if and only if $w(c_0)\ge w(f)+2h$, where $c_0$ is the $0$-th coefficient of the $\Phi$-adic development of $f(x)$.
Now, the coefficient $c_0$ is the $0$-th coefficient of the $\Phi$-adic development of $b_0$.
We can express this coefficient as: $c_0=d_0+u_0$. By (\ref{bounds2}), it is sufficient to check that $w(d_0)\ge w(f)+2h$.

From (\ref{bounds2}) we also have: $v\left(d_1(\alpha)\right)\ge(w(f)+h-V)/e$. Since $v\left(\Phi(\alpha)\right)=v\left(A(\alpha)\right)\ge(V+h)/e$, we get,
$v\left(d_1(\alpha)\Phi(\alpha)\right)\ge(w(f)+2h)/e$. Therefore, $v(d_0(\alpha))\ge (w(f)+2h)/e$.
By Lemma \ref{wv}, this implies $w(d_0)\ge w(f)+2h$, and condition (2) holds.

Suppose now that condition (2) holds. By Lemma \ref{eqwphi},
  $w(\Phi)=V$ and:
$$
\begin{array}{l}
v\left(\Phi(\t)\right)=(V+h_\Phi)/e\ge (V+2h)/e,\\
v\left(\phi(\t)\right)=(V+h_\phi)/e\ge (V+h)/e.
\end{array}
$$
Hence, $v\left(A(\t)\right)\ge(V+h)/e$, and since $\deg A<m$, we have $w(A)\ge V+h$, by Lemma \ref{wv}. Thus, all bounds (\ref{bounds}), (\ref{bounds2}) hold. Let $c_0=d_0+u_0$ be, as above, the $0$-th coefficient of the $\Phi$-adic development of $f(x)$. By hypothesis, $w(c_0)\ge w(f)+2h$, and by  (\ref{bounds2}), $w(u_0)\ge w(f)+2h$; hence, $w(d_0)\ge w(f)+2h$, so that $v(d_0(\t))\ge(w(f)+2h)/e$. On the other hand, by  (\ref{bounds2}) we have also $v(d_1(\t))\ge (w(f)+h-V)/e$, so that
$$
v\left(d_1(\t)\Phi(\t)\right)\ge (w(f)+h-V)/e + (V+2h)/e=(w(f)+3h)/e.
$$
Hence, $v\left(a_0(\t)+a_1(\t)A(\t)\right)=v\left(d_0(\t)+d_1(\t)\Phi(\t)\right)\ge(w(f)+2h)/e$, and condition (3) holds.

Finally, if we exchange the roles of $\alpha$ and $\t$
(i.e. exchange the roles of $\phi$ and $F$), the above arguments also show that condition (3) implies condition (1).
\end{proof}

Along the proof of the theorem we got some precise information about the coefficient $c_1$ of the canonical $\Phi$-development of $f(x)$.

\begin{corollary}\label{ac}
The coefficient $c_1\in\zpx$ of the canonical $\Phi$-adic development of $f(x)$ sa\-tisfies:
$w(c_1-a_1)\ge w(a_1)+h$.
\end{corollary}

\begin{proof}
Clearly, $c_1=d_1+u_1+a_1+g_1$, and by (\ref{bounds2}), the three elements $d_1$, $u_1$, $g_1$ have $w$-value greater than or equal to $w(f)+h-V=w(a_1)+h$.
\end{proof}

Furthermore $a_0(x)a_1^{-1}(x)$ has integral coefficients:

\begin{corollary}\label{polynomial}
The quotient $a_0(\alpha)/a_1(\alpha)$ belongs to the order $\oo[\alpha]\subseteq \oo_M$.
\end{corollary}

\begin{proof}
The choice $A(x)=\phi(x)-F(x)$ obviously sa\-tisfies Theorem \ref{main} (2). Hence,
$$
v\left(a_0(\alpha)+a_1(\alpha)A(\alpha)\right)\ge (w(f)+2h)/e.
$$
Take $\rho:=(a_0(\alpha)/a_1(\alpha))+A(\alpha)$. By (\ref{vai}), $v(\rho)\ge (V+2h)/e$. Theorem \ref{theoexp} shows that $v(\rho)>\exp(\phi)$, so that $\rho$ belongs to $\oo[\alpha]$, and $a_0(\alpha)/a_1(\alpha)$ belongs to $\oo[\alpha]$ too.
\end{proof}


\subsection{Newton inversion modulo a power of the maximal ideal}\label{subsecInversion}
As a consequence of Theorem \ref{main}, every iteration of the single-factor lifting algorithm should efficiently compute
a polynomial $A(x)\in\zpx$, such that $\deg A<m$ and
$$
v\left(a_0(\alpha)+a_1(\alpha)A(\alpha)\right)\ge (w(f)+2h)/e,
$$
where $\alpha$ is a root of $\phi(x)$.
By the argument given in the proof of Corollary \ref{polynomial},
this polynomial $A(x)$ always exists.
A possible solution would be to invert the element $a_1(\alpha)$ in the field $M=K(\alpha)$
and consider the polynomial $A(x)$ such that $A(\alpha)=-a_0(\alpha)/a_1(\alpha)$.
However, for polynomials of large degree, or having large coefficients,
the application of an extended GCD algorithm usually leads to an explosion of coefficients.

Instead, we shall compute an approximation to $-a_0(\alpha)/a_1(\alpha)$
modulo a sufficiently high power of $\m_M$, by applying classical Newton lifting.

By Lemma \ref{lemPsi} we can construct a polynomial $\Psi(x)\in K[x]$ of degree less than $m$ with $w(\Psi)=-w(a_1)=-w(f)+V$.
If we set
$$
A_0(x):=a_0(x)\Psi(x)\bmod \phi(x), \quad A_1(x)=a_1(x)\Psi(x)\bmod\phi(x),
$$
then $v(A_1(\alpha))=0$ and
\[
v(A_0(\alpha))\ge \frac{w(f)+h}{e}+\frac{-w(f)+V}{e}=\frac{V+h}{e}.
\]
For any polynomial $A(x)\in\zpx$, the following conditions are clearly equivalent:
\begin{enumerate}
\item $v\left(a_0(\alpha)+a_1(\alpha)A(\alpha)\right)\ge (w(f)+2h)/e$,
\item $v\left(A_0(\alpha)+A_1(\alpha) A(\alpha)\right)\ge (w(f)+2h)/e+v(\Psi(\alpha))=(V+2h)/e$,
\item $A(\alpha)\equiv -A_0(\alpha)/A_1(\alpha) \md{(\m_M)^{V+2h}}$.
\end{enumerate}
Because $v(A_0(\alpha))\ge (V+h)/e$,
it is sufficient to find an element $A_1^{-1}(\alpha)\in M$ with $A_1^{-1}(\alpha)A_1(\alpha)\equiv 1\mod (\m_M)^{h}$
and then take $A(x)\in K[x]$ to be the unique polynomial of degree less than $m$
satisfying $A(\alpha)=-A_0(\alpha)A_1^{-1}(\alpha)$. By Theorem \ref{theoexp}, we get
$$
v(A(\alpha))=v(A_0(\alpha)A_1^{-1}(\alpha))=v(A_0(\alpha))\ge (V+h)/e>\exp(F)=\exp(\phi),
$$
so that $A(x)\in\zpx$.

We compute the  approximation $A_1^{-1}(\alpha)$ to $A_1(\alpha)^{-1}$ by the classical Newton iteration:
$$
x_{k+1}=x_k(2-A_1(\alpha)x_k),
$$
starting with a lift $x_0\in\oo_M$ of the inverse of $\overline{A_1(\alpha)}$ in the residue field $\oo_M/\m_M$.
Note that if $\op{char}(K)=2$, the iteraton amounts to $x_{k+1}=A_1(\alpha)(x_k)^2$.

This iteration method has quadratic convergence.  If $A_1(\alpha)x_k=1+z$ with $z\in (\m_M)^s$, then $z^2\in (\m_M)^{2s}$ and
$$
A_1(\alpha)x_{k+1}=
\left\{
\begin{array}{ll}
A_1(\alpha) x_k(2-A_1(\alpha)x_k)= (1+z)(1-z)=1-z^2,
&\mbox{ if }\op{char}(K)\ne 2,\\
A_1(\alpha)^2(x_k)^2=(1+z)^2=1+z^2,&\mbox{ if }\op{char}(K)= 2.
\end{array}
\right.
$$
Hence $x_k \equiv A_1(\alpha)^{-1} \mod (\m_M)^{2k}$, which implies that the computation of $A(x)$ requires only $\log_2(h)$ iterations.
Each iteration has a cost of two multiplications (and one addition) in the field $M$.

\subsubsection{Computation of the initial lift}\label{subsecInversionIn}
The efficient computation of an initial lift $x_0\in\oo_M$ of the inverse of $\overline{A_1(\alpha)}$ in $\oo_M/\m_M$ is a non-trivial matter.
Let us explain how to obtain $x_0$ without inverting elements in $M$. Recall the isomorphism
$\gamma\colon \ff{r+1}\lra \oo_M/\m_M$ described in (\ref{gamma}).
As one of the ingredients of a CRT routine on global fields, based on the OM-representations of prime ideals, in \cite[Sec.4.2]{newapp} we described how to compute a section of the reduction mapping:
$$
\oo_M\lra \oo_M/\m_M\stackrel{\gamma^{-1}}\lra \ff{r+1},
$$
For the ease of the reader, we reproduce this description.

Suppose $A_1(x)=g(x)/\pi^\mu$, for some $g(x)\in\oo[x]$. Since $\deg g<m$, the type $\ty$ does not divide $g$: $\ord_\ty(g)=0$. Hence, by \cite[Prop.3.5]{HN} (see also \cite[Prop.2.1]{newapp}), the computation of the residual polynomial of $g(x)$ of $r$-th order yields an identity:
\begin{equation}\label{rp}
\overline{\dfrac{g(\alpha)}{\Phi_r(\alpha)^s\pi_r(\alpha)^u}}=\gamma\left(R_r(g)(z_r)\right)\in\left(\oo_M/\m_M\right)^*,
\end{equation}
where the exponents $s,u$ can be read in $N_r^-(g)$, and $\Phi_r(x),\,\pi_r(x)\in K(x)$ are rational fractions constructed in \cite[Sec.2.4]{HN}, that can be expressed as a products of powers of $\pi,\phi_1,\dots,\phi_r$ with integer exponents:
$$
\Phi_r(x)=\pi^{j_0}\phi_1(x)^{j_1}\cdots \phi_r(x)^{j_r},\quad
\pi_r(x)=\pi^{j'_0}\phi_1(x)^{j'_1}\cdots \phi_r(x)^{j'_r},\quad j_k,\,j'_k\in\Z.
$$ These exponents $j_k,j'_k$ are computed and stored by Montes algorithm. From (\ref{rp}) we deduce:
\begin{align*}
\xi:=\gamma^{-1}(\overline{A_1(\alpha)})&=\gamma^{-1}\left(\overline{g(\alpha)\pi^{-\mu}}\right)=R_r(g)(z_r)
\gamma^{-1}\left(\overline{\pi^\mu\Phi_r(\alpha)^{-s}\pi_r(\alpha)^{-u}}\right)\\&=R_r(g)(z_r)z_1^{t_1}\cdots z_r^{t_r}\in\ff{r+1}^*,
\end{align*}
for some easily computable exponents $t_1,\dots,t_r$ \cite[Lem.1.4]{newapp}. The same lemma may be applied to find integers $t'_1,\dots,t'_r$ such that
$$
\xi':=\gamma^{-1}\left(\overline{\pi_{r+1}(\alpha)^{e\mu}\pi^{-\mu}}\right)=z_1^{t'_1}\cdots z_r^{t'_r}\in\ff{r+1}^*.
$$
Let $\varphi(y)\in\ff{r}[y]$ be the unique polynomial of degree less than $f_r$, such that
$\varphi(z_r)=z_r^{\ell_r e\mu/e_r}(\xi\xi')^{-1}$, and let $\nu:=\ord_y\varphi(y)$. The integer $\ell_r$ satisfies $\ell_rh_r\equiv1\bmod e_r$, and it is also stored by the type $\ty$. The constructive method described in \cite[Prop.2.10]{HN} allows one to compute a polynomial $h(x)\in\oo[x]$ sa\-tisfying the following properties:
$$\deg h(x)<m,\qquad w(h)=e\mu, \qquad y^{\nu}R_r(h)(y)=\varphi(y).
$$
This polynomial satisfies what we want: $\overline{h(\alpha)\pi^{-\mu}}=\gamma(\xi)^{-1}=\overline{A_1(\alpha)}^{-1}$ (cf. loc.cit.). Thus, we may take $x_0=h(\alpha)\pi^{-\mu}$.

\subsection{The main loop}
We are ready to give a detailed description of the iteration steps. Let us recall the preliminary computations before entering into the iteration of the main loop. Suppose $\phi$ is the input Montes approximation to $F$, $\alpha\in \ks$ is a rot of $\phi$, and $M=K(\alpha)$. We compute the first two coefficients $a_0(x),a_1(x)$ of the $\phi$-adic development of $f(x)$, their $w$-value $w(a_0),\,w(a_1)$, and the slope of the $(r+1)$-th order Newton polygon of $f(x)$: $h_\phi=w(a_0)-w(a_1)-V$. Next, we apply the algorithm described in Lemma \ref{lemPsi} to compute the universal polynomial $\Psi(x)\in K[x]$ of degree less than $m$, such that $w(\Psi)=-w(a_1)$. We compute then the polynomials
$$
A_0(x):=a_0(x)\Psi(x) \bmod\phi(x), \quad
A_1(x):=a_1(x)\Psi(x) \bmod\phi(x).
$$
We apply the lifting routine that we just described in the last subsection, to obtain a polynomial $P(x)\in K[x]$, of degree less than $m$, such that  $P(\alpha)A_1(\alpha)\equiv 1\md{(\m_M)^{h_\phi}}$.

Actually, these preliminary computations may be considered the first iteration step. In fact, the next Montes approximation is determined already by:
$$
\Phi:=\phi-A,\quad \mbox{ for } A(x):=-A_0(x)P(x)\bmod\phi.
$$

After the first step, we enter into a general loop. Let $\Phi$ be the $i$-th Montes approximation
to $F$ computed so far, so that $h_\Phi\ge 2h_\phi$, where $\phi$ is the $(i-1)$-th Montes approximation. Let $A:=\phi-\Phi$, $\beta\in\ks$ a root of $\Phi$, $N=K(\beta)$, $\alpha\in\ks$ a root of $\phi$, and $M=K(\alpha)$.\medskip

\noindent{\bf 1. }Compute the first two terms $c_0(x)$, $c_1(x)$,
of the $\Phi$-adic development of $f(x)$. \medskip

\noindent{\bf 2. }$C_0(x):=c_0(x)\Psi(x)\bmod\Phi(x),\quad C_1(x):=c_1(x)\Psi(x)\bmod\Phi(x)$.\medskip

By Corollary \ref{ac} we have $w(c_1)=w(a_1)$; thus, $w(C_1)=0$, or equivalently $v(C_1(\beta))=0$, by Lemma \ref{wv}.
We need now a polynomial $Q(x)\in K[x]$ such that
$$Q(\beta)C_1(\beta)\equiv 1\md{(\m_N)^{h_\Phi}}.
$$
Let $P(x)\in K[x]$ be the analogous polynomial that we used in the previous iteration; with the above notation, $P(x)=A_1^{-1}(x)$ satisfied: $P(\alpha)A_1(\alpha)\equiv 1\md{(\m_M)^{h_\phi}}$. To compute $Q(x)$ we apply a (single!) step of the classical Newton iteration, with $C_1$ replacing $A_1$:  \medskip

\noindent{\bf 3. }$Q(x):=
P(x)(2-C_1(x)P(x))\bmod\Phi(x)$\medskip

Proposition \ref{empalma} below shows that $P(\beta)$ is also an approximation to $C_1(\beta)^{-1}$ modulo $(\m_N)^{h_\phi}$. Thus, $Q(\beta)$ is indeed an approximation to  $C_1(\beta)^{-1}$ with double precision, as required. Finally, we get the next Montes approximation as usual:\medskip

\noindent{\bf 4. }$\Phi':=\Phi-C,\quad \mbox{ for } C(x):=-C_0(x)Q(x)\bmod\Phi$.\medskip

The proof of Proposition \ref{empalma} is based on the following lemma.

\begin{lemma}\label{convert}
With the above notation, let $g(x)\in K[x]$ be a polynomial satisfying $w(g)\ge0$ and $v(g(\alpha))\ge h/e$. Then, $v(g(\beta))\ge h/e$.
\end{lemma}

\begin{proof}
Let $g(x)=\sum_{0\le s}q_s(x)\phi(x)^s$ be the $\phi$-adic development of $g(x)$. By Lemma \ref{wv} (2), $w(q_s\phi^s)\ge w(g)\ge0$, for all $s\ge0$. Since $w(A)\ge V+h$, we get $w(q_sA^s)\ge sh$, for all $s\ge0$.

Since $g(\beta)=\sum_{0\le s}q_s(\beta)A(\beta)^s$, Lemma \ref{wv} (1) shows that
$$
v(q_s(\beta)A(\beta)^s)\ge sh/e, \ \forall\,s>0,\quad v(q_0(\beta))=v(q_0(\alpha))=v(g(\alpha))\ge h/e.
$$
This implies $v(g(\beta))\ge h/e$.
\end{proof}

\begin{proposition}\label{empalma}
With the above notation, let $P(x)\in K[x]$ be a polynomial of degree less than $m$ such that $P(\alpha)A_1(\alpha)\equiv1\md{(\m_M)^h}$. Then,
$P(\beta)C_1(\beta)\equiv 1\md{(\m_N)^{h}}$.
\end{proposition}

\begin{proof}
Since $\deg P<\deg\phi$ we have $w(P)=v(P(\alpha))=0$ by Lemma \ref{wv}.
Also, $w(A_1)=0$ and $w(PA_1-1)\ge 0$.

If we apply Lemma \ref{convert} to the polynomial $g=PA_1-1$, we get $v(P(\beta)A_1(\beta)-1)\ge h/e$.
In particular, $v(P(\beta))=0$.

On the other hand, $w(c_1-a_1)\ge w(a_1)+h$, by Corollary \ref{ac}. Lemma \ref{wv},
shows that $v(c_1(\beta)-a_1(\beta))\ge v(a_1(\beta))+(h/e)$, so that
$$
v(C_1(\beta)-A_1(\beta))=v(c_1(\beta)-a_1(\beta))+v(\Psi(\beta))\ge h/e.
$$
Now, the identity $P(\beta)C_1(\beta)-1=P(\beta)(C_1(\beta)-A_1(\beta))+P(\beta)A_1(\beta)-1$, shows that $v(P(\beta)C_1(\beta)-1)\ge h/e$.
\end{proof}

\section{The Algorithm}\label{secAlgo}
Let $f(x)\in\zpx$ be a monic and separable polynomial, and $\ty$ an $f$-complete optimal type of order $r$, that corresponds to a monic irreducible factor $F(x)\in\zpx$ of $f(x)$. Let $\Phi(x)\in\oo[x]$ be a Montes approximation to $F(x)$. By Lemma \ref{lem45okutsu},
$$
F(x)\equiv \Phi(x)\md{\m^{\lceil\nu\rceil}},\quad\nu=\nu_0+(h_\Phi/e),
$$
where $\nu_0$ is given in (\ref{nu0}) and $e=e_1\dots e_r=e(L/K)$.
So, if $\nu$ is the precision to which we want to find $F$, it is sufficient to
find a Montes approximation $\Phi$ with $h_\Phi\ge e (\nu-\nu_0)$.

We summarize in an algorithm the methods developed in the previous section to achieve this end.
Recall that an initial Montes approximation $\phi(x)$ is always provided by Montes algorithm as an $(r+1)$-th $\phi$-polynomial: $\phi:=\phi_{r+1}$.
As before we set $w:=v_{r+1}$. The function ``quotrem'' returns the quotient and remainder of its parameters.

\Algo{Single-Factor Lifting}{algaccel}
{$f\in\oo[x]$ monic separable,
$\ty$ an $f$-complete optimal type corresponding to some monic irreducible factor  $F(x)\in\oo[x]$ of $f(x)$,
$\phi\in\oo[x]$ a representative of $\ty$,
$\nu\in\N$ a desired precision.}
{An irreducible polynomial $\Phi\in\oo[x]$ such that $\Phi\equiv F\bmod\m^\nu$}
{
\begin{enumerate}
\item[(1)] $a, a_0 \gets \quotrem(f,\phi)$, $a_1 \gets a \bmod \phi$
\item[(2)] $h_\phi\gets w(a_0)-w(a_1\phi)$
\item[(3)] Find $\Psi\in K[x]$ with $\deg\Psi<\deg\phi$ and $w(\Psi)=-w(a_1)$ (cf. Lemma \ref{lemPsi})
\item[(4)] $A_0\gets \Psi a_0\bmod\phi$, $A_1\gets\Psi a_1\bmod\phi$
\item[(5)] Find $A_1^{-1}\in K[x]$ with $w\left((A_1^{-1}A_1\bmod\phi)-1\right)>0$
(cf. Section \ref{subsecInversionIn})
\item[(6)] $s\gets 1$
\item[(7)] while $s<h_\phi$: \textbf{(Newton inversion)}
\begin{itemize}
\item[(a)] $A_1^{-1}\gets A_1^{-1}(2-A_1 A_1^{-1})\bmod\phi$
\item[(b)] $s\gets 2s$
\end{itemize}
\item[(8)] $A\gets A_0A_1^{-1}\bmod\phi$, $\Phi\gets\phi+A$, $C_1^{-1}\gets A_1^{-1}$
\item[(9)] $h\gets h_\phi$
\item[(10)] while $h<e(\nu -\nu_0)$: \textbf{(The main loop)}
\begin{itemize}
\item[(a)]  $c,c_0 \gets \quotrem(f,\Phi)$, $c_1 \gets c \bmod \Phi$
\item[(b)]  $C_0\gets \Psi c_0\bmod\Phi$, $C_1\gets\Psi c_1\bmod\Phi$
\item[(c)]  $C_1^{-1}\gets C_1^{-1}(2-C_1 C_1^{-1})\bmod\phi$
\item[(d)]  $C\gets C_0C_1^{-1} \bmod\Phi$
\item[(e)]  $\Phi\gets\Phi+C$
\item[(f)]  $h\gets 2h$
\end{itemize}
\item[(11)] return $\Phi$
\end{enumerate}
}

Note that the output is always an irreducible polynomial in $\oo[x]$, regardless of the quality of the prescribed precision $\nu$. Of course, if $\nu$ is too small, the output polynomial will not be necessarily irreducible modulo $\m^\nu$.

Algorithm \ref{algaccel} can be simplified by removing the Newton inversion loop.
Then the main loop is entered with $h=1\le h_\phi$ and the initial approximation $A_1^{-1}$ for $C_1^{-1}$ computed in step (5).
This avoids the computation of $w(a_0)$ in step (2)  but
comes with the additional cost of computing more remainders $c_0$ and $c_1$.
We get:

\Algo{Short Single-Factor Lifting} {algaccelshort}
{$f\in\oo[x]$ monic separable,
$\ty$ an $f$-complete optimal type corresponding to some monic irreducible factor  $F(x)\in\oo[x]$ of $f(x)$,
$\phi\in\oo[x]$ a representative of $\ty$,
$\nu\in\N$ a desired precision.}
{An irreducible polynomial $\Phi\in\oo[x]$ such that $\Phi\equiv F\bmod\m^\nu$}
{
\begin{enumerate}
\item[(1)] $a,a_0 \gets \quotrem(f,\phi)$, $a_1 \gets a \bmod \phi$
\item[(2)] Find $\Psi\in K[x]$ with $\deg\Psi<\deg\phi$ and $w(\Psi)=-w(a_1)$ (cf. Lemma \ref{lemPsi})
\item[(3)] $A_0\gets \Psi a_0\bmod\phi$, $A_1\gets\Psi a_1\bmod\phi$
\item[(4)] Find $C_1^{-1}\in K[x]$ with $w\left((C_1^{-1}A_1\bmod\phi)-1\right)>0$
(cf. Section \ref{subsecInversionIn})
\item[(5)]  $\phi\gets\phi+\left(A_0 C_1^{-1} \bmod\phi\right)$
\item[(6)] $h\gets 2$
\item[(7)] while $h<e(\nu -\nu_0)$: \textbf{(The main loop)}
\begin{itemize}
\item[(a)]  $c,c_0 \gets \quotrem(f,\phi)$, $c_1 \gets c \bmod \phi$
\item[(b)]  $C_0\gets \Psi c_0\bmod\phi$, $C_1\gets\Psi c_1\bmod\phi$
\item[(c)]  $C_1^{-1}\gets C_1^{-1}(2-C_1 C_1^{-1})\bmod\phi$
\item[(d)]  $C\gets C_0C_1^{-1} \bmod\phi$
\item[(e)]  $\phi\gets\phi+C$
\item[(f)]  $h\gets 2h$
\end{itemize}
\item[(8)] return $\phi$
\end{enumerate}
}

In the following we restrict our analysis to Algorithm \ref{algaccelshort}. In practice, Algorithm  \ref{algaccel} has a better average performance than Algorithm \ref{algaccelshort}.

\subsection{Precision}
The precision necessary to perform the computations in each step of the algorithm
is relevant for the complexity analysis and for efficiently implementing the algorithm.
It is most efficient to conduct each computation with a fixed precision, say $\mu$; that is, we truncate
the $\pi$-adic expansion of all elements in $\oo$ after the $\mu$-th $\pi$-adic digit.
This precision is increased in each iteration of the loop.

We analyze the precision needed in the main loop by going through the steps in reverse order.
By Theorem \ref{main}, Lemma \ref{eqwphi}) and Corollary \ref{polynomial}, the polynomial $C(x)$ computed in step (5d) has coefficients in $\oo$, and it is expected to satisfy:
\[
v(\phi(\t)+C(\t))\ge\frac{2h_\phi+V}{e}.
\]
Thus, in (5e), we need to know the coefficients of $C(x)\in \zpx$ to a $\pi$-adic precision of
$\left\lceil(2h_\phi+V)/e\right\rceil$ digits.

We denote by $\exp(F)$ the exponent of the polynomial $F$ (see Theorem \ref{theoexp}).
As for all polynomials $B(x)\in K[x]$ that occur in the algorithm
the element $B(\t)$ is integral,
they can be represented in the form $B(x)=b(x)/\pi^d$ where $b(x)\in\zpx$ and $0\le d\le\exp(F)$. So the loss of precision in each multiplication in steps
(5b), (5c), and (5d) is at most $\exp(F)$ $\pi$-adic digits.
Thus the needed precision for $C(x)$ can be guaranteed if $c_0(x)$ and $c_1(x)$ are computed with a $\pi$-adic precision of
$\left\lceil(2h_\phi+V)/e\right\rceil+4\exp(F)$ digits. To this purpose,
it is sufficient to conduct the division with remainder with this precision.

\begin{lemma}\label{lemprec}
If all polynomials in the main loop in Algorithm \ref{algaccelshort}
are represented in the form $b(x)/\pi^d$ where $b(x)\in\zpx$ and $0\le d\le\exp(F)$,
a $\pi$-adic precision of $\left\lceil\frac{2h+V}{e}\right\rceil+4\exp(F)$
for the numerator is sufficient in each iteration of the main loop.
\end{lemma}

\subsection{Complexity of single-factor lifting}

In the following we give a complexity estimate for the steps in the algorithm, assuming that the residue field $\ff{}$ is finite. Let $n=\deg f$, $m=\deg F=\deg\phi$, and $R=\op{depth}(F)$.

\begin{itemize}
\item[(1)] The divisions with remainder can be conducted in $O((n-m)m)$ operations in $\oo$.
\item[(2)] By \cite[Lem.4.21]{HN}, the computation of $w(a_1)=v(a_1(\t))$ is essentially equivalent to the computation of the $(\phi_1,\dots,\phi_r)$-multiadic expansion of $a_1$. By \cite[Lemma 18]{pauli10} it takes $m^2$ operations in $\oo$ to compute $w(a_1)$.
\item[(3)]
The polynomial $\Psi$ with $w(\Psi)=-w(a_1)$ is constructed as $\Psi(x)=\pi^{j_\pi}\phi_1^{j_1}\dots\phi_R^{j_R}$, for exponents $j_\pi,j_1,\dots,j_R$ that can be found in $O((\log m)^3)$ integer operations of integers less than
           $m$ by Lemma \ref{lemPsi}.  The power product  needed for computing $\Psi(x)$
           can be evaluated in  $O(m^2)$ operations in $\oo$.
\item[(4)] Two polynomials of degree up to $m$ can be multiplied in $O(m^2)$ operations in $\oo$,
           the reduction by the polynomial $\phi$ also takes $O(m^2)$ operations in $\oo$.
\item[(5)] By \cite[section 9]{pauli10} a polynomial representation of the initial value of $C_1^{-1}$ can be
           found in $O(m^2(\log m)^2)$ operations in $\oo$.
\item[(6)] There are $\log_2(e(\nu-\nu_0))$ iterations of the main loop.
           In each iteration there are two divisions with remainder that take
           $O((n-m)m)$ operations in $\oo$.
           Furthermore, the iteration requires five multiplications and two additions;
these operations, including the reduction by the polynomial $\phi(x)$, take
           $O(m^2)$ operations in $\oo$.  So in total each iteration of the loop consists of $O(nm)$ operations in $\oo$.
\end{itemize}

If we so do not take the necessary $\pi$-adic precision into account we obtain:

\begin{lemma}
Let $K$ be a local field with finite residue field, $\oo$ its valuation ring and $f(x)\in\oo[x]$ a monic separable polynomial of degree $n$.
Algorithm \ref{algaccelshort} can lift a Montes approximation $\phi(x)\in\oo[x]$ to an irreducible factor $F(x)\in\oo[x]$ of degree $m$
of $f(x)$, to a precision of $\nu$ $\pi$-adic digits, in $O\left(n m[(\log m)^2+\log(e\nu)]\right)$ operations in $\oo$, where $e$
is the ramification index of $K[x]/(F(x))$ over $K$.
\end{lemma}

In the special case $K=\Q_p$ we include the cost of the operations in $\Z_p$ in our complexity estimate.
In our estimates we assume that two $p$-adic numbers of precision
$\nu$ can be multiplied in
$O(\nu\log \nu\log\log \nu)=O(\nu^{1+\epsilon})$
operations of integers less than $p$ \cite{schoenhage-strassen}.

Because it is our goal to give a complexity estimate for polynomial factorization in general and
the cost of steps (1), (2), (3), and (4) is included in the complexity estimate of Montes algorithm
we only consider the main loop in the next lemma.

\begin{lemma}
Let $f(x)\in\Z_p[x]$ be a monic separable polynomial of degree $n$. Algorithm \ref{algaccelshort} can lift a Montes approximation $\phi(x)\in\Z_p[x]$ to an irreducible factor
$F(x)$ of degree $m$ of $f(x)$, to a $p$-adic precision of $\nu$ digits, in
$O\left(n m[\nu^{1+\epsilon}+v(\disc(F))^{1+\epsilon}]\right)$ operations of integers less than $p$ in the main loop.
\end{lemma}

\begin{proof}
Let $L=\Q_p[x]/(F(x))$, and let $e$ be the ramification index of $L/\Q_p$.
By Lemma \ref{lemprec} the precision needed in the $j$-th iteration ($1\le j<\log_2(e\nu)$) of the main loop is
\[
\left\lceil\frac{2^j+V}{e}\right\rceil+4\exp(F)
\le
\left\lceil\frac{2^j+5V}{e}\right\rceil,
\]
the last inequality by Theorem \ref{theoexp}. Let $s=\lceil\log_2(e\nu)\rceil$. Clearly, for $\epsilon\ll0$, we have
$$
\sum_{0\le j<s}2^\delta=\left\{
\begin{array}{ll}
 O(2^s),&\mbox{ for }\delta=1+\epsilon,\\
O(s),&\mbox{ for }\delta=\epsilon
\end{array}
\right.
$$
Now, the number of operations of integers less than $p$ in the main loop is approximately
\begin{align*}
e^{-(1+\epsilon)}\sum_{1\le j<s} (2^j+5V)^{1+\epsilon}
&\le e^{-(1+\epsilon)}\sum\nolimits_{1\le j<s} \bigl(2^{j(1+\epsilon)}+2^{j\epsilon}5V+2^j(5V)^{\epsilon}+(5V)^{1+\epsilon}\bigr)\\
&=e^{-(1+\epsilon)}O\left(2^{(1+\epsilon)s}+sV+2^sV^{\epsilon}+s V^{1+\epsilon}\right)\\
&=O\left(\nu^{1+\epsilon}+e^{-\epsilon}s(V/e)+\nu(V/e)^{\epsilon}+s(V/e)^{1+\epsilon}\right)\\
&=O\left(\nu^{1+\epsilon}+s(V/e)^{1+\epsilon}\right),
\end{align*}
the last equality because $e^{-\epsilon}s(V/e)$ is dominated by $s(V/e)^{1+\epsilon}$ and $\nu(V/e)^{\epsilon}$ is dominated by either $\nu^{1+\epsilon}$ or $s(V/e)^{1+\epsilon}$. By Theorem \ref{theoexp},
$$
V/e=\exp(F)+\nu_0\le 2\exp(F)\le 2\ind(F)\le v(\disc(F)).
$$
On the other hand, $\log e\le \log v(\disc(L))\le \log v(\disc(F))$, so that $(\log e)v(\disc(F))^{1+\epsilon}=O(v(\disc(F))^{1+\epsilon})$. Therefore, the term $s(V/e)^{1+\epsilon}=O(\log(e\nu)v(\disc(F))^{1+\epsilon})$ is dominated either by $\nu^{1+\epsilon}$ (if $\nu\ge v(\disc(F))$) or by $v(\disc(F))^{1+\epsilon})$ (if $\nu<v(\disc(F))$).
This ends the proof of the lemma.
\end{proof}

\subsection{Complexity of Polynomial Factorization over $\Z_p[x]$}

The complexity estimates for Montes algorithm \cite{ford-veres,pauli10}
are based on \cite[Proposition 4.1]{pauli01}, which asserts that
if $n v(\phi(\t)) > 2v(\disc(f))$ for all roots $\t$ of $f(x)$ and if the degree of $\phi(x)$
is less than or equal to the degree of any irreducible factor of $f(x)$, then $f(x)$ is irreducible.
Because the improvement of the approximation $\phi(x)$ to an irreducible factor of $f(x)$
measured by $v(\phi(\t))$ is at least $2/n$ in each step,
Montes algorithm determines whether a polynomial is irreducible in at most
$v(\disc(f))$ steps.  A detailed analysis of the algorithm yields:

\begin{theorem}[{\cite[Theorem 1]{pauli10}}]
Let $p$ be a fixed prime.  We can establish whether a polynomial
$f(x)\in\Z_p[x]$ of degree $n$ is irreducible in at most $O(n^{2+\epsilon}v(\disc(f))^{2+\epsilon})$
operations of integers less than $p$.
\end{theorem}

If $f(x)$ is reducible, Montes algorithm finds a $\phi_i(x)$ such that $N_{i}^-(f)=N_{\phi_i,v_i}^-(f))$
consists of more than one segment in less than $v(\disc(f))$ iterations.
Each of these segments corresponds to a factor $g(x)$ of $f(x)$ and
Montes algorithm branches to find improved approximations to each of these factors based on $\phi_i(x)$.
Now, by \cite[Proposition 4.1]{pauli01},
the irreducibility of $g(x)$ can be
determined or the algorithm comes across a Newton polygon whose principal part consists of more than one segment
in less than $v(\disc g)$ steps.
Thus, since $v(\disc(gh))\ge v(\disc(g))+v(\disc(h))$ for all polynomials $g(x)$ and $h(x)$,
$v(\disc(f))$ is also an estimate for the number of steps needed to find Montes approximations
to all irreducible factors of $f(x)$.  We get:

\begin{corollary}
Let $p$ be a fixed prime. Montes approximations to all irreducible factors of
$f(x)\in\Z_p[x]$ of degree $n$  can be found in at most $O(n^{2+\epsilon}v(\disc(f))^{2+\epsilon})$
operations of integers less than $p$.
\end{corollary}

Let $m_1,\dots,m_k$ denote the degrees of the irreducible factors $F_1,\dots,F_k$ of $f(x)$.
As $\sum_{i=1}^k m_i=n$
the Montes approximations of all factors can be lifted to a precision of $\nu$ $p$-adic digits in
\[
\sum_{i=1}^k
O\left(n m_i[\nu^{1+\epsilon}+v(\disc(F_i))^{1+\epsilon}]\right)
=
O\left(n^{2} [\nu^{1+\epsilon}+v(\disc(f))^{1+\epsilon}]\right)
\]
operations of integers less than $p$. Thus, we find the following general estimation for the complexity of the factorization algorithm that combines Montes algorithm with the single-factor lifting algorithm.

\begin{theorem}
Let $p$ be a fixed prime, $f(x)\in\Z_p[x]$ a polynomial of degree $n$, and $\nu\in\N$ a prescribed precision.
One can find approximations $\Phi(x)\in\Z_p[x]$ to all irreducible factors $F(x)$ of $f(x)$, with $F(x)\equiv\Phi(x) \mod p^\nu$,
in at most $O(n^{2+\epsilon}v(\disc f)^{2+\epsilon}+n^{2}\nu^{1+\epsilon})$ operations of integers less than
$p$.
\end{theorem}

\subsection{Direct single-factor lifting}

Let $f(x)\in\oo[x]$ and assume we know a monic factor $\overline\phi(x)\in\ff{}[x]$
of $\overline{f}(x)\in\ff{}$ such that ${\overline\phi}^2\nmid\overline f$. By Hensel lemma, there is a unique irreducible factor $F(x)\in\oo[x]$ of $f(x)$ whose reduction modulo $\m$ is $\overline\phi(x)$.
In this case, any monic lift $\phi(x)\in\zpx$ of $\overline\phi(x)$ is already a Montes approximation to $F(x)$, with respect to the type of order zero determined by $\overline\phi(x)$. We can use the single-factor lifting algorithm directly without any prior iterations of Montes algorithm.
If we specialize Algorithm \ref{algaccel} accordingly we obtain:

\Algo{Direct Single-Factor Lifting}{algpure}
{$f\in\oo[x]$, $\overline\phi\in\ff{}[x]$ irreducible such that $\overline\phi\mid\overline f$
but $\overline\phi^2\nmid\overline f$, $\nu\in\N$}
{An irreducible polynomial $\Phi\in\zpx$ dividing $f$ modulo $\pi^\nu$, such that $\overline\Phi=\overline\phi$}
{
\begin{enumerate}
\item  $a, a_0 \gets \quotrem(f,\phi)$, $a_1 \gets a \bmod \phi$
\item $h_\phi\gets v_1(a_0)$
\item  Find $a_1^{-1}\in\oo[x]$ such that $\overline{a}_1 \overline{a_1^{-1}}\equiv 1 \bmod \overline\phi$
\item for $1\le i\le\lceil\log_2(h_\phi)\rceil$: $a_1^{-1}\gets
 a_1^{-1}(2-a_1 a_1^{-1})\bmod\phi$
\item  $A\gets a_0a_1^{-1} \bmod\phi$, $\Phi\gets\phi+A$
\item  for $1\le i< \lceil\log_2(\nu/h_\phi)\rceil$:
\begin{enumerate}
\item  $a, a_0 \gets \quotrem(f,\Phi)$, $a_1 \gets a \bmod \Phi$
\item  $a_1^{-1}\gets a_1^{-1}(2-a_1 a_1^{-1})\bmod\Phi$
\item  $A\gets a_0a_1^{-1} \bmod\Phi$, $\Phi\gets\Phi+A$
\end{enumerate}
\item return $\Phi$
\end{enumerate}
}

The valuation $v_1$ of step (2) was defined in section \ref{secTypes}: $v_1(a_0)$ is the minimum of the $v$-values of the coefficients of $a_0$. The computation of the initial value of $a_1^{-1}$ in step (3) is trivial now; it amounts to compute a section of the ring homomorphism $\oo[x]\lra \oo[x]/(\pi,\phi)$.
The $\pi$-adic precision required in each iteration of the first loop is $2^i$ digits.
In the second loop we need a precision of $2^{i+1}h_\phi$ digits.
It is easy to see that the complexity of Algorithm \ref{algpure} is the same as the complexity of
the quadratic Hensel Lift algorithm \cite{zassenhaus}. In practice, Algorithm \ref{algpure} has a slightly better performance.

\section{Experimental results}\label{secNumerical}

The combination of  algorithm \ref{algaccel} with Montes algorithm yields a new $p$-adic polynomial factorization algorithm. We have implemented this algorithm in {\tt Magma} to check its practical efficiency; the implementation can be obtained from {\url {http://themontesproject.blogspot.com}}.
Our routine, called {\tt SFLFactor}, takes a separable monic  polynomial $f\in\Z[x]$, a prime number $p$ and a certain precision $\nu$ and returns $p$-adically irreducible polynomials $\phi_1,\dots,\phi_m\in\Z[x]$ such that $f\equiv \phi_1\dots\phi_m\pmod{p^\nu}$.

Besides its good theoretical complexity, the routine has a high efficiency in practice. We have applied it to the test polynomials given in the Appendix, and compared the results with those of the standard $p$-adic factorization routines of {\tt Magma}  and  {\tt PARI}. We present here some of these results. All tests have been done in a Linux server, with two Intel Quad Core processors,
running at 3.0 Ghz, with 32Gb of RAM memory. Times are expressed in miliseconds.

\subsection*{Running time vs depth}

The graphic in Figure \ref{table1} shows the running times of our factorization routine  applied to the polynomials $E_{p,j}(x)$ for $p\le 1000$, compared to those of {\tt Magma} and  {\tt PARI}'s functions. {\tt Magma} can't go beyond $j=4$ in less than an hour, while {\tt PARI} reaches only $j=5$; our package takes at most 2 seconds to factor any of these polynomials. The running time of {\tt SFLFactor} on the polynomials $E_{p,8}(x)$ is better observed in Figure \ref{graficr}.

\subsection*{Running time vs width}

The graphic in Figure \ref{table3} compares the behaviour of {\tt SFLFactor},  {\tt Magma} and {\tt PARI} with respect to the width, using the test polynomials $B_{p,k}(x)$ for $k\le 1000$. Since the width tends to be a very pessimistic bound, we have also tested the performance of  {\tt SFLFactor}, with the test polynomials $A_{2,50,50001,r}(x)$, for $1\le r\le 1000$. These polynomials have all the same (large) width, but each one requires $r+1$ iterations of the main loop of Montes algorithm, to detect its $p$-adic irreducibility. Thus, for $r$ large, they constitute very ill-conditioned examples for our algorithm.
The running-times are shown in Figure \ref{table4}.

\begin{center}

\begin{figure}
\caption{\small Running times (in miliseconds) of {\tt SFLFactor } (red), {\tt Magma} (green) and {\tt PARI} (blue).}
\label{table1}
\includegraphics[height=40mm,bb=0 0 480 185]{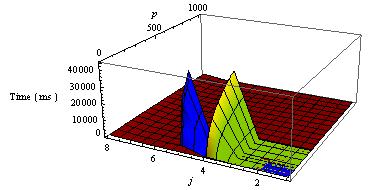}\\
\end{figure}

\begin{figure}
\caption{\small Running times of  {\tt SFLFactor } applied to the polynomials $E_{p,8}(x)$ for $p<1000$.}
\label{graficr}
\includegraphics[height=40mm,bb=0 0 250 210]{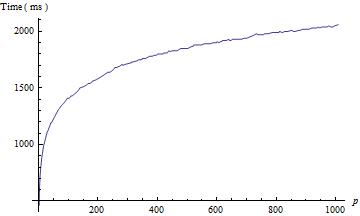}  \\
\end{figure}

\end{center}

\begin{center}
\begin{figure}
\caption{\small   Running times  of {\tt SFLFactor} (red), {\tt Magma}  (green) and {\tt PARI} (blue) applied to the poylnomials $B_{p,k}(x)$}
\label{table3}
\includegraphics[height=40mm,bb=0 0 270 220]{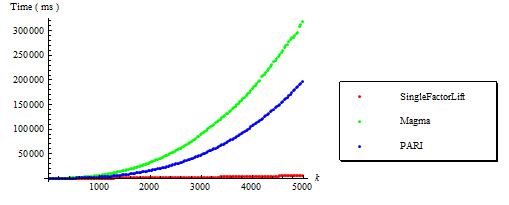}\\
\end{figure}

\begin{figure}
\caption{\small Running times of {\tt SFLFactor} for $A_{2,50,5001,r}(x)$.}
\label{table4}
\includegraphics[height=38mm,bb=0 0 240 220]{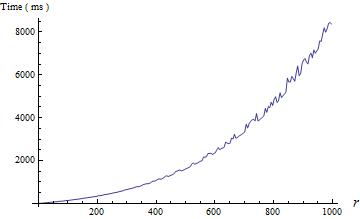}\\
\end{figure}
\end{center}

\subsection*{Running time vs number of factors}

We can observe in Figures \ref{table5} and \ref{table6} the behaviour of {\tt SFLFactor} with respect to the number of factors of the polynomial to be factored. The first graphic  shows the running times of our routine applied to the polynomials $D_{101,p,2,3}(x)$ for the primes $p\in\{1069, 1087, 1091, 1051, 1117, 1097, 919,1009\}$, which cover all the possible splitting types of the $101$-th cyclotomic polynomial.

In Figure \ref{table6} we can compare the performance of our algorithm applied to the polynomials $A_{101,mn,211,0}(x)$ and $A_{101,n,211}^m(x)$. The different height of the polynomials is a plausible explanation for the significative difference in the running times.


\subsection*{Statistical tests}

We have tested algorithm \ref{algpure} to compare its practical performance with that of the classical Hensel lift algorithm. For every $m\in\{2,\dots,20\}$ we have built a list of 1000 random  pairs $\{f,\overline{f}_1\}$, where $f(x)\in\Z[x]$ is a separable product of $m$ quartic irreducible polynomials modulo $17$, and $\overline{f}_1\in\ff{17}[x]$ is a factor of $f$.
For each pair, the factor $\overline{f}_1$  is lifted with both algorithms to $\Z_p[x]$  to precision 50,100,150,\dots,1000 successively. Figure \ref{table-stat} shows the average running times, suggesting that Single-factor lifting seems slightly faster than Hensel lift.

\begin{center}
\begin{figure}
\caption{\small Running times of our package for $D_{101,p,2,3}(x)$, $p\in\{1069, 1087, 1091, 1051, 1117, 1097, 919,1009\}$.}
\label{table5}
\includegraphics[height=45mm,bb=0 0 240 200]{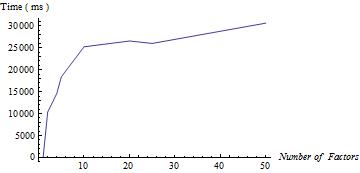}       \\
\end{figure}
\end{center}

\begin{center}
\begin{figure}
\caption{\small Running times of our package for $A_{101,mn,211,0}(x)$ and $A_{101,n,211}^m(x)$.}
\label{table6}
\includegraphics[height=45mm,bb=0 0 320 280]{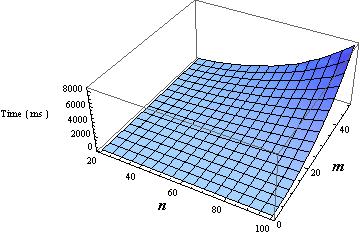}\qquad\qquad
\includegraphics[height=45mm,bb=0 0 320 280]{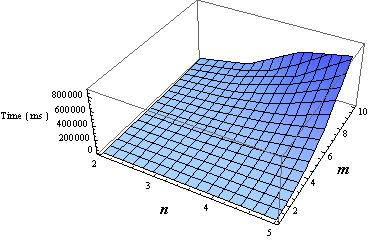}

\end{figure}
\end{center}


\begin{center}
\begin{figure}
\caption{\small Average running times of  statistical tests on Single-factor Lifting (Green) and Hensel Lift  (Red).}
\label{table-stat}
\includegraphics[height=40mm,bb=0 0 540 240]{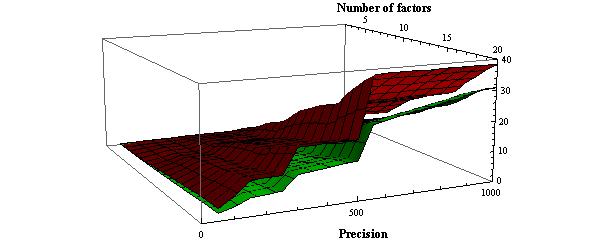}   \\
\end{figure}

\end{center}


\section*{Appendix: Families of test polynomials}

Along the design of a new algorithm, it is useful to dispose of a bank of benchmarks to test its efficiency. Different authors (\cite{cohen}, \cite{ford-pauli-roblot}) have provided such benchmarks for different problems in computational algebraic number theory. These lists of polynomials have been of great use, but the new algorithms and the fast evolution  of  hardware have  left it out of date.
We propose an update consisting of several parametric families of polynomials, which should cover all the  computational difficulties one may encounter in problems concerning prime ideals in number fields (prime ideal factorization, $p$-adic factorization, computation of $p$-integral bases, etc).

Classically, it has been considered that the invariants of an irreducible polynomial $f(x)\in\Z[x]$ that determine its computational complexity are the degree, the height (maximal size of the coefficients) and, when we focus on a prime number $p$, the $p$-index. The $p$-index of $f(x)$ is the $p$-adic valuation of the index $\left(\Z_K\colon \Z[\t]\right)$, where $\t\in\overline{\Q}$ is a root of $f(x)$, $K=\Q(\t)$ and $\Z_K$ is the ring of integers of $K$. The $p$-index is closely related to the $p$-adic valuation of the discriminant $\disc(f)$.

As mentioned in section \ref{secWidth}, for a finer analysis of the complexity two more invariants must be taken into account: the depth and width of the different $p$-adic irreducible factors of $f(x)$. Therefore, our families of test polynomials are described in terms of different integer parameters which affect its degree, height, index, number of $p$-adic irreducible factors, and their depth and width.
The computational complexity of the aforementioned problems can be adjusted to the reader's convenience by a proper choice of the parameters, by combining different issues or focusing on a concrete one.

The test polynomials are gathered in Table \ref{Tab1}. The parameters appearing in the table may be required to satisfy particular conditions in each family.

\begin{table}
\caption{Families of test polynomials}\label{Tab1}
\begin{tabular}{|l|}
\hline
$A_{p, n, k,r}(x)=(x+1+p+\dots+p^r)^n+p^{k};$\\ [1ex]
\hline
$A_{p, n, k}^m(x)=(x^n+2p^{k})((x+2)^n+2p^{k})\dots ((x+2m-2)^n+2p^{k})+2p^{mnk};$\\[1ex]
\hline
$B_{p, k}(x)=(x^2-2x+4)^3+p^k;$\\[1ex]
\hline
$C_{p,k}(x)= \left((x^6+4 p\, x^3+3 p^2 x^2+4 p^2)^2 +p^6\right)^3+p^{k}$\\[1ex]
\hline
$D_{\ell, p, n, k}(x)=(x^{\ell-1}+x^{\ell-2}+\dots+x+1)^n+p^k$\\[1ex]
\hline
$E_{p, 1}(x)=x^2+p$\\[1ex]
\hline
$E_{p, 2}(x)=E_{p,1}(x)^{2}+(p-1)p^{3}x$\\[1ex]
\hline
$E_{p, 3}(x)=E_{p,2}(x)^{3}+p^{11}$\\[1ex]
\hline
$E_{p, 4}(x)=E_{p,3}(x)^{3}+p^{29}x E_{p,2}(x)$\\[1ex]
\hline
$E_{p, 5}(x)=E_{p,4}(x)^{2}+(p-1)p^{42}xE_{p,1}(x)E_{p,3}(x)^2$ \\[1ex]
\hline
$E_{p, 6}(x)=E_{p,5}(x)^{2}+p^{88}xE_{p,3}(x)E_{p,4}(x)$\\[1ex]
\hline
$E_{p, 7}(x)=E_{p,6}(x)^{3}+p^{295}E_{p,2}(x)E_{p,4}(x)E_{p,5}(x)$\\[1ex]
\hline
$E_{p, 8}(x)=E_{p,7}(x)^{2}+(p-1)p^{632}xE_{p,1}(x)E_{p,2}(x)^2E_{p,3}(x)^2E_{p,6}(x)$\\[1ex]
\hline
\end{tabular}
\end{table}

The main characteristics of these polynomials are  summarized  in Table \ref{Tabsideways}. The notation used in the headers of the table is:
\begin{itemize}
\item[] $\op{depth}_p(f):=$ maximum depth of the $p$-adic irreducible factors of $f$.
\item[] $\op{width}_p(f):=$sum of the components of the widths of all the local factors of $f$.
\item[] $\ind_p(f):=p$-adic valuation of the index of $f$.
\item[] $\Delta_p(K):= p$-adic valuation of the discriminant of the number field $K$ defined by $f$.
\item[] $p\Z_K=$ factorization of the prime $p$ in the ring of integers of $K$. A term $\p_{f}^e$ means a prime ideal with ramification index $e$ and residual degree $f$ (no exponent or subindex are written if they are~1).
\end{itemize}
Further explanations about each family  are given in the subsequent subsections.

It is worth mentioning that the polynomials in our list can be combined to build new examples of test polynomials, whose characteristics will combine those of the factors. The philosophy is: take $f,g$ from the table and form the polynomial $h=fg+p^a$, with $a\in\N$ high enough. Indeed, this is the technique used to build the polynomials $A_{p,n,k}^m(x)$ and $D_{\ell, p, n, k}(x)$. This procedure allows everyone to build its own test polynomial with local invariants at her convenience.\medskip

A final remark concerning the use of our test polynomials: they are not only intended to compare the performance of different algorithms. They are also useful to analyse  the influence of the different parameters in your favourite algorithm. Besides the obvious tests between polynomials in the same family,  more subtle comparisons can be done to study the performance of your algorithm. The following table proposes some of them:\bigskip

\begin{center}
\begin{tabular}{|c|c|l|}
\hline
$f$ & $g$  & useful to check dependency on \\
\hline
\hline
$E_{p,4}(x)$ &$C_{p,28}(x)$  & number of factors\\
\hline
$C_{p,k}(x)$ & $A_{p,36,k,0}(x) $&  depth\\
\hline
$D_{\ell,p,n,k}(x)$ & $A_{p,n(\ell-1),k}(x)$ & width\\
\hline
$A_{p,n,k,0}(x)$ & $A_{p,n,k,k-1}(x)$ & precision\\
\hline
\end{tabular}
\end{center}\bigskip

\noindent{\bf Notation.} From now on, whenever we deal with a prime number $p$, we denote by $v_p$ the $p$-adic valuation of $\Z_p$ normalized by $v_p(p)=1$. \medskip

\begin{sidewaystable}
\vskip 120mm
\caption{Characteristics of the test polynomials}\label{Tabsideways}
\as{1.5}
$$
\begin{array}{|c|c|c|c|c|c|c|c|}
\hline
f & \deg f   &
\begin{tabular}{c}  \#\mbox{ $p$-adic factors}\end{tabular} & \op{depth}_p(f) &
\op{width}_p(f) & \ind_p(f)& \Delta_p(K) & p\Z_K \\
\hline
A_{p, n, k, r} & n & 1 & 1 & \lceil k/n\rceil   & (k-1)(n-1)/2 & nv_p(n)+n-1& \p^n\\
\hline
A_{p, n, k}^m & nm &m  & 1 & m\lceil k/n\rceil  & m(k-1)(n-1)/2 & m(nv_p(n)+n-1) &\p^n\stackrel{(m}\cdots\p^n\\
\hline
B_{p, k} & 6 & 2 & 1 &2\lceil k/3\rceil&  2(k-1) & 4 & \p^3\p^3 \\
\hline
C_{p,k} & 36 & 6 & 3& 6k-90   &12k+78 & 24& \p^3_{2}\p^3_{2}\p^3_{2}\p^3_{2}\p^3_{2}\p^3_{2}\\
\hline
D_{\ell, p, n, k} & n(\ell -1) & g:={(\ell-1})/{\operatorname{ord}_{\ff{\ell}^\ast} (p)} & 1 & g\lceil k/n \rceil  & (n-1)(\ell-1)(k-1)/2 & (\ell-1)(nv_p(n)+n-1) & \p_{\frac{\ell-1}g}^n\stackrel{(g}\dots\p_{\frac{\ell-1}g}^n
 \\
\hline
E_{p,3}& 12 & 1 & 3 & 3 & 52  & 11 & \p^{12}\\
\hline
E_{p,4}& 36 & 1 & 4 & 4 & 553  & 35& \p^{36}\\
\hline
E_{p,5}& 72 & 1 & 5 & 5 & 2300  & 71 & \p^{72}\\
\hline
E_{p,6}& 144 & 1 & 6 & 6 &9378  & 143 &\p^{144}\\
\hline
E_{p,7}& 432 & 1 & 7 & 7  & 85476  & 431& \p^{432}\\
\hline
E_{p,8}& 864 & 1 & 8 & 8 & 342981  &  863 & \p^{864}\\
\hline
\end{array}
$$
\end{sidewaystable}

\subsection*{Family 1: $p$-adically irreducible polynomials of depth 1 and large index}\mbox{\null}

Let $p$ be a prime number. Take two coprime integers $n,k\in\N$, and $r\in\{0,1,\dots,\lfloor k/n\rfloor\}$. Define:
\vskip 2mm
\begin{center}
\fbox{$\displaystyle
A_{p,n,k,r}(x)=(x+1+p+p^2+\dots+p^r)^n+p^{k}
$
}\vskip 2mm
\end{center}
Our test polynomial is obtained from $A_{p,n,k}:=x^n+p^k$ by a linear change of the variable: $x\mapsto x+1+p+\dots+p^r$. Hence, these two polynomials have the same discriminant:
$$
\disc(A_{p,n,k,r})=\disc(A_{p,n,k})=(-1)^{n(n-1)/2}n^{n}p^{(n-1)k}.
$$

\noindent{\bf Proposition A1. }\it
Let $K_{p, n, k,r}$ be the number field defined by a root of $A_{p, n, k,r}(x)$.
\begin{itemize}
\item[a)] $\ind_p(A_{p, n, k,r})=(k-1)(n-1)/2$.
\item[b)] $v_p(\disc(K_{p,n, k,r} ))=nv_p(n)+n-1$.
\item[c)] $p\Z_{K_{p, n, k,r}}=\p^n,$ where  $\p$ is a prime ideal of residual degree~1.
\item[d)] The $p$-adically irreducible polynomial $A_{p, n, k,r}(x)$ has depth~1 and width  $(\lceil k/n \rceil)$.
\end{itemize}\rm

\begin{proof}
Take $\phi(x)=x+1+p+\dots+p^r$. The Newton polygon of first order $N_{\phi,v_p}(A_{p, n, k,r})$ is one-sided, with end points $(0,k)$, $(n,0)$, and slope $-k/n$. Thus, the prime $p$ is totally ramified in $K_{p,n,k,r}$. Proposition \ref{indextheorem} gives immediately the value of the index of $A_{p,n,k,r}$:
$$
\ind_p(A_{p,n,k,r})=(k-1)(n-1)/2.
$$
Hence, $v_p(\disc(K_{p,n,k,r}))=v_p(\disc(A_{p,n,k,r}))-2\ind_p(A_{p,n,k,r})=nv_p(n)+n-1$.
\end{proof}

For $k\le n$, these polynomials may have large degree and index, but they have small width (equal to $1$). For $k\gg n$ they have large width too.
In the latter case, the parameter $r$ may have an influence on the speed of an algorithm to save the obstruction of the high width. For instance, Montes algorithm performs $r+1$ iterations of its main loop before reaching the polynomial $\phi$ considered in the proof of Proposition A1, as an optimal lift to $\Z[x]$ of the irreducible factor $x+1$ of $A_{p,n,k,r}(x)$ modulo $p$.

\subsection*{Family 2: Arbitrary  number of  depth~1 $p$-adic factors and large index}\mbox{\null}


Let  $p>3$ be a prime number. Take $n,k$ coprime positive integers such that $k>nv_p(n)$, and  $m$ any integer such that $1<m<p/2$.
 Define:

\vskip 2mm
\fbox{ $\displaystyle A_{p, n, k}^m(x)= (x^n+2p^{k})((x+2)^n+2p^{k})\dots ((x+2m-2)^n+2p^{k})+2p^{mnk} $
}
\vskip 2mm
This polynomial is irreducible over $\Q$, since it is $2$-Eisenstein.\medskip

\noindent{\bf Lemma A2. }\it The $p$-valuation of the discriminant of $A_{p,n,k}^m(x)$ is:\rm
$$
v_p(\disc(A_{p,n,k}^m ))=  m(nv_p(n)+k(n-1)).
$$

\begin{proof}
The discriminant of $A(x):=x^n+2p^k$ is $(-1)^{n(n-1)/2}n^{n}2^{n-1}p^{(n-1)k}$. Take $F(x)=A(x)A(x+2)\dots A(x+2m-2)$; since all these factors of $F(x)$ are coprime modulo $p$:
$$
v_p(\disc (F))=m v_p(\disc(A))=m(nv_p(n)+k(n-1)).
$$
From $A_{p,n,k}^m=F+2p^{mnk}$, we get $v_p(\disc (A^m_{p,n,k}))=v_p(\disc(F))$, because $mnk>v_p(\disc(F))$, by our assumption on $k$.
\end{proof}

\noindent{\bf Proposition A3. }\it Let $K_{p,n,k}^m$ be the number field defined by a root of $A_{p, n, k}^m(x)$
\begin{itemize}
\item[a)]
$\ind_p(A_{p,n,k}^m)=m(k-1)(n-1)/2$.
\item[b)] $v_p(\disc(K_{p,n,k}^m))=m(nv_p(n)+n-1)$.
\item[c)] $\displaystyle p\Z_{K_{p,n,k}^m}=\p_1^n\cdots\p_m^n$, all prime ideals with residual degree $1$.
\item[d)] The $m$ $p$-adic factors of $A_{p, n, k}^m(x)$ have depth~1 and width  $(\lceil k/n \rceil)$.
\end{itemize}\rm

\begin{proof}Let $A(x)=x^n+2p^k$, and $\phi(x)=x$. Clearly $A_{p, n, k}^m(x)=a(x)\phi(x)^n+b(x)$, where $a(x)=A(x+2)\cdots A(x+2m-2)$ and $b(x)=2p^ka(x)+2p^{mnk}$. Since $a(x)$ is not divisible by $x$ modulo $p$, this $\phi$-development of $A_{p, n, k}^m$ is \emph{admissible} \cite[Def.1.11]{HN}, and it can be used to compute the principal Newton polygon of the first order $N^-_{\phi,v_p}(A_{p, n, k}^m)$ \cite[Lem.1.12]{HN}. Since $v_p(a(x))=0$ and $v_p(b(x))=k$, this polygon is one-sided of slope $-k/n$. Hence, $A_{p, n, k}^m(x)$ has a $p$-adic irreducible factor of degree $n$, depth $1$, index $(k-1)(n-1)/2$ and width $(\lceil k/n\rceil)$, which is congruent to a power of $x$ modulo $p$, and determines a totally ramified extension of $\Q_p$. The same argument, applied to $\phi_j(x)=x+2j$, for $1\le j<m$, determines all other irreducible factors of $A_{p, n, k}^m(x)$. Since these factors are pairwise coprime modulo $p$, the index of $A_{p, n, k}^m(x)$ is $m$ times the index of each local factor.
This proves all statements of the proposition.
\end{proof}


\subsection*{Family 3: Low degree, two $p$-adic factors of depth~$1$, and large width and index}\mbox{\null}

For $p\equiv1\pmod3$  a prime number and $k\in \N$, $k\not\equiv0\pmod{3}$, define the  polynomial
\vskip 2mm
\begin{center}
\fbox{$\displaystyle
B_{p,k}(x)=(x^2-2x+4)^3+p^{k}
$
}\vskip 2mm
\end{center}
\vskip 2mm
This polynomial is irreducible over $\Q$. In fact, it has two irreducible cubic factors over $\Z_p$ (by the proof of the proposition below) and it it is the cube of a quadratic irreducible factor modulo 3. The discriminant of $B_{p,k}(x)$ is
$$
\disc(B_{p,k})=-2^6 3^6 p^{4 k} \left( p^k+27\right).
$$

\noindent{\bf Proposition A4. }\it
Let $K_{p,k}$ be the number field defined by a root of the polynomial $B_{p,k}(x)$.
\begin{itemize}
\item[a)] $\ind_p(B_{p,k})=2(k-1)$.
\item[b)] $v_p(\disc(K_{p,k} ))=4$.
\item[c)] $p\Z_{K_{p,k}}=\p^3\p'^3,$ where  $\p, \p'$ are prime ideals of residual degree~1.
\item[d)] The two $p$-adic factors of $B_{p, k}(x)$ have depth~1 and width  $(\lceil k/3 \rceil)$.
\end{itemize}\rm

\begin{proof}
Let $x^2-2x+4=\phi_1(x)\phi_2(x)$ be the factorization of $x^2-2x+4$ in $\Z_p[x]$, into the product of two monic linear factors. Since these factors are coprime modulo $p$, the expression $B_{p,k}(x)=(\phi_1(x))^3(\phi_2(x))^3+p^k$ is simultaneously an admissible $\phi_i$-expansion of $B_{p,k}$, for $i=1,2$ \cite[Def.1.11]{HN}, and we can use this development to compute the Newton polygons of the first order $N^-_{\phi_i,v_p}(B_{p,k})$, for $i=1,2$ \cite[Lem.1.12]{HN}. Both polygons are one-sided of slope $-k/3$ and end points $(0,k)$, $(3,0)$. This proves c) and d).

On the other hand, Proposition \ref{indextheorem} shows that $\ind_p(\phi_1)=\ind_p(\phi_2)=k-1$. Since $\phi_1$ and $\phi_2$ are coprime modulo $p$, this proves a) and b).
\end{proof}

\subsection*{Family 4: Six $p$-adic factors of depth~3,  fixed medium degree, and large index}\mbox{\null}


Let $p\equiv 5\pmod{12}$ be a prime number. Take an integer $k>18$ and define:
\vskip 2mm
\begin{center}
\fbox{$\displaystyle
C_{p,k}(x):=\left((x^6+4 p\, x^3+3 p^2 x^2+4 p^2)^2 +p^6\right)^3+p^{k}.
$
}\vskip 2mm
\end{center}

\noindent{\bf Proposition A5. }\it
Suppose that $C_{p,k}$ is irreducible over $\Q$, and let $K_{p,k}$ be the number field generated by one of its roots.
\begin{itemize}
\item[a)] $\ind_p(C_{p,k})=12k+78;$
\item[b)] $v_p(\disc(K_p))=24;$
\item[c)] $p\Z_{K_p}=\p^3_{1}\cdots  \p^3_{6},$
all prime ideals $\p_j$ with residual degree $2$.
\item[d)] The six $p$-adic factors of $C_p(x)$ have depth~3 and width  $(1,1,k-17)$.
\end{itemize}\rm

\begin{proof}The proof consists of an application of Montes algorithm by hand. We leave the details to the reader. The algorithm outputs six $C_{p,k}$-complete strongly optimal types of order $3$. Three of them have the following fundamental invariants $(\phi_i,\lambda_i,\psi_i)$ at each level $i$:
$$
(y;(x,-1/3,y+2);(\phi_2,-1,y^2+3);(\phi_3+ip^3,17-k,y-\omega)),
$$
where $\phi_2(x)=x^3+2p$, $\phi_3(x)=x^6+4px^3+3p^2x^2+4p^2$, $i\in\Z$ satisfies $i^2\equiv-1\md{p^{k-17}}$ and $\omega\in\ff{p^2}^*$ runs on the three cubic roots of $-i(-2)^{6-k}\in\ff{p}^*$. The other three complete types are obtained by replacing $i$ by $-i$.

The Theorem of the index \cite[Thm.4.18]{HN} shows that $\ind_p(C_{p,k})=12k+78$. The computation of $v_p(\disc(K_{p,k}))$ is trivial, since $p$ is tamely ramified.
\end{proof}

\subsection*{Family 5: Large degree, multiple $p$-adic factors of depth~1 and large index and width}\mbox{\null}


Let  $\ell, p$ be two different prime numbers and $n, k \in \N$ two coprime integers. Consider the polynomial:
\vskip 2mm
\begin{center}
\fbox{$\displaystyle
D_{\ell,p,n,k}(x):=\Phi_{\ell}(x)^n+p^k,
$
}\vskip 2mm
\end{center}
where $\Phi_{\ell}(x)=1+x+\dots+x^{\ell-1}$ is the $\ell$-th cyclotomic polynomial.
\medskip

\noindent{\bf Lemma A6. }\it
The $p$-valuation of the discriminant of $D_{\ell, p, n, k}$ is:\rm
$$
v_p(\disc(D_{\ell, p, n, k}))=(\ell-1)(nv_p(n)+k(n-1)).
$$

\begin{proof}
Let $\alpha_1,\dots,\alpha_{\ell-1}$ be the roots of $\Phi_\ell(x)$, and $\beta_1,\dots,\beta_{\ell-2}$ the roots of
$\Phi'_\ell(x)$. Write $d=\deg D_{\ell,p,n,k}=n(l-1)$.
$$
\begin{array}{rl}
\disc(D_{\ell,p,n,k})&\displaystyle  =(-1)^{d(d-1)/2}\operatorname{Res}(\Phi_{\ell}(x)^n+p^k, n\Phi_{\ell}(x)^{n-1}\Phi'_{\ell}(x))\\
&\displaystyle =(-1)^{d(d-1)/2}n^d(l-1)^d \prod\nolimits_{\alpha_i}(\Phi_\ell(\alpha_i)^n+p^k)^{n-1}\prod\nolimits_{\beta_i}(\Phi_\ell(\beta_i)^n+p^k)\\
&\displaystyle = (-1)^{d(d-1)/2}n^dp^{k(\ell-1)(n-1)}(l-1)^d\prod\nolimits_{\beta_i}(\Phi_\ell(\beta_i)^n+p^k).\\
\end{array}
$$
The term $(l-1)^d\prod_{\beta_i}(\Phi_\ell(\beta_i)^n+p^k)$ is congruent, up to a sign,  to $\disc(\Phi_\ell)^n$ modulo $p$; thus, it is not divisible by $p$ and the conclusion of the lemma follows.
\end{proof}

\noindent{\bf Proposition A7. }\it Assume that the polynomial $D_{\ell, p, n, k}(x)$ is irreducible over $\Q$ and
let $K_{\ell, p, n, k}$ be the number field generated by one of its roots.
Denote by $f$ the order of $p$ in the multiplicative group $\ff{\ell}^\ast$, and set $g=(\ell-1)/f$.
\begin{itemize}
\item[a)] $v_p(\ind(D_{\ell, p, n, k}))=(\ell-1)(n-1) (k-1)/2$.
\item[b)] $v_p(\disc(K_{\ell, p, n, k}))=(\ell-1)(nv_p(n)+n-1)$.
\item[c)] $p\Z_{K_{\ell, p, n, k}}= \p_{1}^n\cdots\p_{g}^n$, all prime idals $\p_j$
 with residual degree $f$.
\item[d)] The $g$ $p$-adic factors of $D_{\ell, p, n, k}(x)$ have depth~1 and width  $(\lceil k/n\rceil)$.
\end{itemize}\rm

\begin{proof}
The cyclotomic polynomial $\Phi_\ell$ splits in $\Z_p[x]$ into the product $\Phi_\ell=\phi_1\cdots \phi_g$, of $g$ irreducible factors of degree $f$. Since these factors are coprime modulo $p$, the expression $D_{\ell,p,n,k}=(\phi_1)^n\cdots(\phi_g)^n+p^k$ is simultaneously an admissible $\phi_i$-expansion of $D_{\ell,p,n,k}$, for all $1\le i\le g$ \cite[Def.1.11]{HN}, and we can use this development to compute the $g$ Newton polygons of the first order $N^-_{\phi_i,v_p}(D_{\ell,p,n,k})$ \cite[Lem.1.12]{HN}. All these polygons are one-sided of slope $-k/n$ and end points $(0,k)$, $(n,0)$. This proves c) and d).

On the other hand, Proposition \ref{indextheorem} shows that $\ind_p(\phi_i)=f(n-1)(k-1)/2$, for all $i$. Since $\phi_1,\dots,\phi_g$ are coprime modulo $p$, we have $\ind_p(D_{\ell,p,n,k})=g\ind_p(\phi_1)=gf(n-1)(k-1)/2$. This proves a) and b).
\end{proof}

With a proper election of the primes $\ell, p$ we can achieve arbitrarily large values of $f$ and $g$, with the only restriction $fg=\ell-1$.

\subsection*{Family 6: $p$-adically irreducible polynomials of fixed large degree and depth}\mbox{\null}


For any prime number $p>3$, consider the following polynomials:
\begin{center}
$$
\begin{array}{|l|}
\hline\displaystyle E_{p,1}(x)=x^2+p\\ [1ex]
\hline\displaystyle E_{p,2}(x)=E_{p,1}(x)^{2}+(p-1)p^{3}x\\[1ex]
\hline\displaystyle E_{p,3}(x)=E_{p,2}(x)^{3}+p^{11}\\[1ex]
\hline\displaystyle E_{p,4}(x)=E_{p,3}(x)^{3}+p^{29}xE_{p,2}(x)\\[1ex]
\hline\displaystyle E_{p,5}(x)=E_{p,4}(x)^{2}+(p-1)p^{42}xE_{p,1}(x)E_{p,3}(x)^2 \\ [1ex]

\hline\displaystyle E_{p,6}(x)=E_{p,5}(x)^{2}+p^{88}xE_{p,3}(x)E_{p,4}(x)\\[1ex]
\hline\displaystyle E_{p,7}(x)=E_{p,6}(x)^{3}+p^{295}E_{p,2}(x)E_{p,4}(x)E_{p,5}(x)\\[1ex]
\hline\displaystyle E_{p,8}(x)=E_{p,7}(x)^{2}+(p-1)p^{632}xE_{p,1}(x)E_{p,2}(x)^2E_{p,3}(x)^2E_{p,6}(x)\\[1ex]
\hline\end{array}
$$
\vskip 2mm
\end{center}

These polynomials have been built recursively through a constructive application of Montes algorithm. They are all irreducible over $\Z_p$ and determine totally ramified extensions of $\Q_p$. The depth of $E_{p,i}$ is $i$, and an Okutsu frame is given by $[\phi_1=x,\phi_2=E_{p,1},\dots,\phi_i=E_{p,i-1}]$. The Newton polygons $N_i(E_{p,j})$, for $j\ge i$, are one-sided of slope $\lambda_i$, where:
$$
\lambda_1=-\dfrac{1}2,\quad \lambda_2=-\dfrac{3}2,\quad \lambda_3=\lambda_4=-\dfrac{2}3,\quad
\lambda_5=\lambda_6=-\dfrac{1}2,\quad \lambda_7=-\dfrac{1}3,\quad \lambda_8=-\dfrac{1}2.
$$
The values of $\ind_p(E_{p,i})$ are given in Table \ref{Tabsideways}; they have been derived from Proposition \ref{indextheorem}.

\subsection*{Families of test equations for function fields} Let $\mathcal{F}$ be a perfect field, and $p$ an indeterminate. One checks easily that all polynomials of Table \ref{Tab1} are irreducible over $\mathcal{F}[p]$; hence, they may be used to test arithmetically oriented algorithms for function fields.

\end{document}